\documentclass[10pt]{amsart}

\usepackage{amsmath}
\usepackage{amsthm}
\usepackage{amscd} 
\usepackage{amssymb}
\usepackage{amsrefs}
\usepackage{comment}
\usepackage{comment}
\usepackage{tikz, tkz-euclide} 
\usetikzlibrary{decorations.pathreplacing}
\usetikzlibrary{matrix}

\newcommand{\punctured}{\mathbb{CP}^1\setminus\{a_1,\ldots,a_{n}\}}
\newcommand{\Npunctured}{\mathbb{CP}^1\setminus\{a_1=0,a_2,\ldots,a_{n(N)}\}}
\newcommand{\qsec}{q_{_2}}
\newcommand{\qk}{q_{_{k}}}
\newcommand{\f}{\mathbf{f}}
\newcommand{\B}{\mathbf{B}}
\newcommand{\C}{\mathbf{c}}

\newcommand{\tB}{\tilde{B}}
\newcommand{\tc}{\tilde{c}}
\DeclareMathOperator{\Aut}{Aut}
\renewcommand{\Re}{\operatorname{Re}}
\renewcommand{\d}{\operatorname{d}}
\renewcommand{\q}{q_{_N}}

\theoremstyle{theorem}
\newtheorem{theorem}{Theorem}[section]
\newtheorem{prop}[theorem]{Proposition}
\newtheorem{coro}[theorem]{Corollary}
\newtheorem{lemma}[theorem]{Lemma}

\theoremstyle{definition}
\newtheorem{definition}[theorem]{Definition}
\newtheorem{example}[theorem]{Example}

\theoremstyle{theorem}
\newtheorem{Theorem}{Theorem}
\newtheorem{Coro}[Theorem]{Corollary}

\theoremstyle{remark}
\newtheorem{remark}{Remark}

\newtheorem{claim}{Claim}

\numberwithin{equation}{section}

\title[Punctured Riemann sphere and Modular functions]{Hyperbolic Metric, punctured Riemann sphere and Modular functions}
\author{Junqing Qian}
\address{Department of Mathematics\\Storrs\\CT 06268}
\curraddr{}
\email{junqing.qian@uconn.edu}
\thanks{This work was supported by NSF grant DMS-1611745.}
\subjclass[]{}
\date{}
\dedicatory{}

\begin{document}
	
	\begin{abstract}
		We derive a precise asymptotic expansion of the complete K\"{a}hler-Einstein metric on the punctured Riemann sphere with three or more omitting points. By using Schwarzian derivative, we prove that the coefficients of the expansion are polynomials on the two parameters which are uniquely determined by the omitting points. Futhermore, we use the modular form and Schwarzian derivative to explicitly determine the coefficients in the expansion of the complete K\"{a}hler-Einstein metric for punctured Riemann sphere with $3, 4, 6$ or $12$ omitting points.
	\end{abstract}
	\maketitle
    \tableofcontents
    \section{Introduction}\label{sec:introduction}

The main object in this paper is the punctured Riemann sphere $\punctured$, where $\mathbb{CP}^1$ is the projective space over $\mathbb{C}$ of dimension one, and $\{a_1,\ldots,a_n\}$ are $n$ different points that are omitting from $\mathbb{CP}^1$. It is well-known that there exists a unique complete K\"{a}hler-Einstein metric on $\mathbb{CP}^1\setminus\{a_1,\ldots,a_n\}$, $n\ge 3$, with negative constant Gauss curvature. However, an explicit asymptotic expansion of the metric near $a_j$,  $j=1,\ldots,n$, remains unknown. More precisely, the coefficients in the expansion are difficult to determine. In this article, we derive a precise asymptotic formula for the K\"{a}hler-Einstein metric whose coefficients are polynomials on the first two parameters, which are determined by the punctures $\{a_1,\ldots,a_n\}$. Furhtermore, all coefficient polynomials can be explicitly written down when $n=3,4,6,12$.

The asymptotic expansion of the complete K\"ahler-Einstein metric on the quasi-projective manifold $M = \overline{M} \setminus D$ was proposed by Yau~\cite[p. 377]{YauReviewOfGeometry}, where $D$ is a normal crossing divisor in the project manifold $\overline{M}$ such that $K_{\overline{M}} + D > 0$. The leading order term has been known to him since the late 1970s (see \cite{Yau1978}). Several people have worked on the asymptotic expansion, see for example \cite{Schumacher1998MathAnn, DaminWu2006, RochonZhang2012}, using techniques from partial differential equations. Another important class of the quasi-projective manifolds is that $M$ is the quotient of Siegel space $\mathcal{S}_g/\Gamma$ by an arithmetic subgroup $\Gamma$, $g \ge 2$, and $\overline{M}$ is Mumford's toroidal compactification. In this case $K_{\overline{M}} + D$ is big and nef. The complete K\"ahler-Einstein volume form on $M$ has been written down in \cite{WangSmoothSiegel1993,YauZhang2014}.

The open Riemann surface $\mathbb{C}\mathbb{P}^1 \setminus \{a_1, \ldots, a_n\}$, $n\ge 3$, is only a one-dimensional example of $\overline{M} \setminus D$ with $K_{\overline{M}} + [D] > 0$. It is nevertheless the building block of a general complete quasi-projective manifold with negative holomorphic sectional curvature. Indeed, given a projective manifold $X$, for any point $x \in X$, there exists a Zariski neighborhood $U = X \setminus Z$ (where $Z$ is an algebraic subvariety of $X$) of $x$ such that $U$ can be embedded into the product 
$$ S_1 \times \cdots \times S_N $$
as a closed algebraic submanifold, in which $N$ is some positive integer, and each $S_j$ is of the form $\mathbb{C}\mathbb{P}^1 \setminus \{a_1, \ldots, a_n\}$ for $n \ge 3$ (see \cite[p. 25, Lemma 2.3]{GriffithsZariski1971}). Take the complete K\"ahler-Einstein metric on each $S_j$. Then, the product metric restricted to $U$ gives a complete K\"ahler metric $\omega$ on $U$ of finite volume and negatively pinched holomorphic sectional curvature. Furthermore, by a recent result \cite[Theorem 3]{DaminYau2017}, the quasi-projective manifold $U$ possesses a complete K\"ahler-Einstein metric which is uniformly equivalent to the metric $\omega$; see also \cite[Example 5.4]{DaminYau2018}. 
 
To derive the expansion of K\"ahler-Einstein metric on $\mathbb{C}\mathbb{P}^1 \setminus \{a_1, \ldots, a_n\}$, we 
make use of the modular forms in number theory. This technique enables us to obtain the precise expansion with explicit coefficients, which were obscure in the literature by the abstract series expansion or the kernel of local linear operators. Recall the well-known uniformization theorem (see \cite{HubbardTeichmuller}) which indicates that the punctured Riemann sphere has the unit disk $\mathbb{D}$ as its universal covering space, or equivalently, the upper half plane $\mathbb{H}$. The idea is to find this covering map, then induce the Poincar\'{e} metric from $\mathbb{H}$ to our object. 

F. Klein and R. Fricke had included classic theories regard automorphic function in their lecture notes (see \cite{KleinBook}). Later, the Teichm\"{u}ller space was introduced by O. Teichm\"{u}ller and developed by L. Ahlfors and L. Bers (see \cite{BersTechmuller, BersUniModuliKlein, AhlforsTeichmuller}). One of the methods that turns out very useful is the application of Schwarzian derivative. Let $f$ be a covering map for the universal covering space $\mathbb{H}\rightarrow\punctured$, then $f$ uniquely determines its Schwarzian derivative 
$$\{f,\tau\}=2(f_{\tau\tau}/f_{\tau})_{\tau}-(f_{\tau\tau}/f_{\tau})^2,\quad\mbox{where }f=f(\tau),\,\tau\in\mathbb{H}.$$
On the other hand, the Schwarzian derivative of the inverse of $f$ is well-defined and uniquely determined by the following equation
\begin{equation}\label{eq:intro-schwarzian-expansion}
\{\tau,f\}=\sum_{j=1}^{n}\left(\frac{1}{(f-a_j)^2}+\frac{2\beta_j}{\alpha_j}\frac{1}{f-a_j}\right),
\end{equation}
where all $a_j$ are the omitting points from the punctured Riemann sphere, and $\beta_j, \alpha_j$ are some constants, $j=1,\ldots,n$, which will be discussed in section \ref{section:schwarzian-derivative}. By an argument from its geometric property, there are $(n-3)$ ratios from the set $\{\frac{\beta_1}{\alpha_1},\ldots,\frac{\beta_n}{\alpha_n}\}$ that are independent to each other, these constants are often called the accessory parameters. Determining the accessory parameters is a way to figure out the automorphic function. This approach has been studied by I. Kra in \cite{IrwinKraAccessoryPara} and A. B. Venkov in \cite{VenkovRussianPaper}. However, despite the results from automorphic function and Teichm\"{u}ller theory, such covering map $f:\mathbb{H}\rightarrow\punctured$ is still quite difficult to write out.

On the other hand, the congruence subgroup in $\mbox{SL}_2(\mathbb{Z})$ is a great collection of Fuchsian groups with nice properties. J. McKay and A. Sebbar discovered a connection between modular forms and Schwarzian derivative of automorphic functions in \cite{McKay2000}, which will be mentioned in section \ref{section:modular-forms}. Furthermore, there are several interesting connections between modular curve, Schwarzian derivative and graph theory that are mentioned by A. Sebbar in \cite{SebbarTorsionFree, SebbarModularCurve}.


The main result of this article is given by the following theorem.
\begin{Theorem}\label{thm:into-main-1}
	A universal covering map $f:\mathbb{H}\rightarrow\mathbb{CP}^1\setminus\{a_1=0,a_2,\ldots,a_n\}$ which vanishes at infinity can be given by the following expansion
	\begin{equation}
	f=f(\tau)=A\left(\qk+\frac{B}{A}\qk^2+\C_3(A,\frac{B}{A})\qk^3+\sum_{m=4}^{\infty}\C_m(A,\frac{B}{A})\qk^m\right)\notag
	\end{equation}
	in $\qk=\exp\{\frac{2\pi i}{k}\tau\}$, $\tau\in\mathbb{H}$, for some real number $k$, where $A,B$ are constants depending only on $a_2,\ldots,a_n$, and the coefficient term $\C_m(A,\frac{B}{A})$ is a polynomial on $A$ and $\frac{B}{A}$ for $m\ge 3$. In particular, $\C_3(A,\frac{B}{A})=\frac{1}{16}\left[19\frac{B^2}{A^2}-A^2\sum_{j=2}^{n}\left(\frac{1}{a^2_j}-\frac{1}{a_j}\frac{2\beta_j}{\alpha_j}\right)\right]$, where $\alpha_j,\beta_j$ are the constants in equation \eqref{eq:intro-schwarzian-expansion}. Consequently, the complete K\"{a}hler-Einstein metric can be given by the following asymptotic expansion
	\begin{align}
	|ds|=\frac{1}{|f|\log|\frac{f}{A}|}\left|1-\left(\frac{B}{A}\frac{f}{A}-\frac{\Re(\frac{B}{A}\frac{f}{A})}{\log|\frac{f}{A}|}\right)+\sum_{m=2}^{\infty}R_m(A,\frac{B}{A},\frac{f}{A},\frac{f^s\overline{f^{m-s}}}{A^s\overline{A^{m-s}}\log^j|\frac{f}{A}|})\right||df|
	\end{align}
	at the cusp $0$, where $f\in\mathbb{CP}^1\setminus\{a_1=0,a_2,\ldots,a_n\}$ and $R_m(A,\frac{B}{A},\frac{f}{A},\frac{f^s\overline{f^{m-s}}}{A^s\overline{A^{m-s}}\log^j|\frac{f}{A}|})$ is a polynomial in $A,\frac{B}{A},\frac{f}{A},\frac{f^s\overline{f^{m-s}}}{A^s\overline{A^{m-s}}\log^j|\frac{f}{A}|}$, $s,j=0,1,\ldots,m$, with constant coefficients for $m\ge 2$.
\end{Theorem}
To illustrate the second result, we define the following values for $n(N)$,
\begin{equation}\label{eq:cusp-number-Gamma-N}
n(2)=3,\quad n(3)=4,\quad n(4)=6,\quad n(5)=12.
\end{equation}
The second result is given by the following theorem and corollary.
\begin{Theorem}\label{thm:intro-thm-2}
	Let $n(N)$ be the values in equation \eqref{eq:cusp-number-Gamma-N}, and let $f_N:\mathbb{H}\rightarrow\Npunctured$ be a universal covering space with deck transformation group $\Aut(f_N)=\Gamma(N)$-the principal congruence subgroup-which vanishes at infinity,  $N=2,3,4,5$. Then the map $f_N$ can be given as the following expansion 
	\begin{equation}\label{eq:intro-thm-2}
		f_N=f_N(\tau)=A\left( \q+\frac{B}{A}\q^2+\sum_{m=3}^{\infty}\C_{m}(\frac{B}{A})\q^m\right)
	\end{equation}
	in $\q=\exp\{\frac{2\pi}{N}i\tau\},\tau\in\mathbb{H}$, where the constants $A,B\in\mathbb{C}$ are uniquely determined by the set of values of the punctured points $\{a_1=0,a_2,\ldots,a_{n(N)}\}$, and the coefficient term $\C_m(\frac{B}{A})$ is a polynomial in $\frac{B}{A}$ with constant coefficients for $m\ge 3$.
\end{Theorem}

\begin{remark}
	Explicit formulas of $\C_m(A,\frac{B}{A})$ are given in \cite{ChoQianKobayashi} in terms of Bell polynomials.
\end{remark}

In particular, when $n=2$,  $n(2)=3$, it is a triple punctured Riemann sphere, the covering map $\mathbb{H}\rightarrow\mathbb{CP}^1\setminus\{a_1=0,a_2,a_3\}$ for arbitrary $a_2,a_3$ is given in the following example. 

\begin{example}\label{example:intro-thm-2}
	The covering map $f_2:\mathbb{H}\rightarrow\mathbb{CP}^1\setminus\{a_1=0,a_2,a_3\}$ which vanishes at infinity has the following expansion
	\begin{equation*}
	f_2(\tau)=\frac{16a_2a_3}{a_3-a_2}\left[\qsec-\left(8+\frac{16a_2a_3}{a_3-a_2}\right)\qsec^2+\sum_{m=3}^{\infty}\C_m\left(8+\frac{16a_2a_3}{a_3-a_2}\right)\qsec^m\right],
	\end{equation*}
	in $\qsec=\exp\{\pi i\tau\}$, $\tau\in\mathbb{H}$, where $\C_m\left(-8-\frac{16a_2a_3}{a_3-a_2}\right)$ is a polynomial in the term $\left(-8-\frac{16a_2a_3}{a_3-a_2}\right)$ for $m\ge 3$. In particular, $\C_3=\left(-8-\frac{16a_2a_3}{a_3-a_2}\right)^2-20$. In this case, the constants $A,B$ in Theorem \ref{thm:intro-thm-2} are given by the following
	\begin{equation*}
	A=\frac{16a_2a_3}{a_3-a_2},\qquad B=\left(\frac{16a_2a_3}{a_3-a_2}\right)\cdot\left[-\left(8+\frac{16a_2a_3}{a_3-a_2}\right)\right].
	\end{equation*}
\end{example}

\begin{Coro}\label{coro:intro}
	As a consequence of Theorem \ref{thm:intro-thm-2}, the complete K\"{a}hler-Einstein metric on $\Npunctured$ has the following asymptotic expansion
	\begin{equation}
	|ds|=\frac{1}{|f|\log|\frac{f}{A}|}\left|1-\left(\frac{B}{A}\frac{f}{A}-\frac{\Re(\frac{B}{A}\frac{f}{A})}{\log|\frac{f}{A}|}\right)+\sum_{m=2}^{\infty}R_m(\frac{B}{A},\frac{f}{A},\frac{f^s\overline{f^{m-s}}}{A^s\overline{A^{m-s}}\log^j|\frac{f}{A}|})\right||df|
	\end{equation}
	at the cusp $a_1=0$, where $R_m(\frac{B}{A},\frac{f}{A},\frac{f^s\overline{f^{m-s}}}{A^s\overline{A^{m-s}}\log^j|\frac{f}{A}|})$ is a polynomial in $\frac{B}{A},\frac{f}{A},\frac{f^s\overline{f^{m-s}}}{A^s\overline{A^{m-s}}\log^j|\frac{f}{A}|}$, $s,j=0,1,\ldots,m$, with constant coefficients for $m\ge 2$.
\end{Coro}

As an application of our result, several examples are given in the section \ref{section:example}. For example, we give the complete K\"{a}hler-Einstein metric of $\mathbb{CP}^1\setminus\{0,1,\infty\}$ in the following example.

\begin{example}\label{example:intro}
	The covering map $f_2:\mathbb{H}\rightarrow\mathbb{CP}^1\setminus\{0,1,\infty\}$ has the following expansion
	\begin{equation}
	f=f_2(\tau)=16\qsec-128\qsec^2+704\qsec^3-3072\qsec^4+11488\qsec^5+O(\qsec^6)
	\end{equation}
	in $\qsec=\exp\{\pi i\tau\}$, $\tau\in\mathbb{H}$, and we write $f=f_2$ for convenience. In this case, the constants $A,B$ in Theorem \ref{thm:intro-thm-2} take values $A=16, B=-128$. Therefore the coefficients of $f=f_2$ in equation \eqref{eq:intro-thm-2} are givien by $B/A=\frac{-128}{16}=-8$, $\C_3=(B/A)^2-20=(-8)^2-20=44$. From Corollary \ref{coro:intro}, the complete K\"{a}hler-Einstein metric has the following explicit expansion
	\begin{align*}
	|ds|&=\frac{1}{|f|\log|f/16|}\left|1+\frac{1}{2}\left(f-\frac{\Re f}{\log|f/16|}\right)\right.\\
	&\qquad\left.-\left[\frac{51}{32}f^2-\frac{1}{4}\frac{f\Re f}{\log|f/16|}+\frac{51}{64}\frac{\Re(f^2)}{\log|f/16|}+\frac{1}{4}\frac{(\Re f)^2}{\log^2|f/16|}\right]+O(f^3)\right||df|
	\end{align*}
	at the cusp $0$, where $f\in\mathbb{CP}^1\setminus\{0,1,\infty\}$.
\end{example}

\begin{remark}
		The K\"ahler-Einstein metric on $\mathbb{CP}^1\backslash\{0,1,\infty\}$ was known by S. Agard. He gave a globle explicit formula in terms of a double integral (see equation (2.6) in \cite{AgardDist}).
\end{remark}

\begin{remark}
        The argument of this article can derive a general case of $a_1,a_2,a_3$, please see Corollaries \ref{coro:section-8-1} and \ref{coro:section-8-2}.
\end{remark}

In this article, we start with studying the fundamental group $\pi_1$ of the punctured Riemann sphere, its topology indicates that every element from $\pi_1(\punctured)$ corresponds to an automorphism of its universal cover, more precisely an element in $\mbox{SL}_2(\mathbb{R})$. Two similar arguments from its monodromy action imply Theorem \ref{thm:generator-are-parabolic}, which indicates that each generator is of \textit{parabolic} type. This fact allows us to interpret the covering map $f$ as an expansion at the corresponding cusp, so as the complete K\"{a}hler-Einstein metric. In the next section, section \ref{section:schwarzian-derivative}, we introduce and discuss the Schwarzian derivative in the analytic sense, and derive Theorem \ref{thm:coefficient-depend-on-A-B} at the end, which plays an important role to our main result. In section \ref{section:ramification-point}, we will focus on the action of a discrete group on $\mathbb{H}$. Couple propositions are introduced in section \ref{section:modular-group} in order to target the congruence subgroups that are of genus \textit{zero} without \textit{elliptic} point. In section \ref{section:modular-forms}, we introduce the modular forms and discuss its relation with the Schwarzian derivative. The main results and couple examples are provided in the last section.
    \section{Construction of Universal Covering Space}\label{section:construction of universal covering space}

We will sketch the process of the universal covering from $\mathbb{H}$ to $\punctured$ in the graphic way. The general procedure of constructing the universal covering space is mentioned in \cite[p. 8-12, section 2]{NevaROLF}. I will describe the most straightforward case in this article, which will be sufficient and helpful for readers to understand the topic. To cover the punctured Riemann sphere $\punctured$ from the unit disk $\mathbb{D}$, or equivalently, the upper half plane $\mathbb{H}$, we first connect the $n$ punctures such that it is a polygon without self-crossing, rename the vertexes as $a_1,a_2,\ldots,a_n$ counter-clockwisely. Then, let us pick $n$ points randomly on the unit circle, and name them $c_1,\ldots,c_n$ counter-clockwisely which correspond to $a_1,\ldots,a_n$ respectively, and let $c_1=1, c_n=-1$. The collection of the $n$ points all lie on the upper half unit circle. Connecting the two points that are right next to each other by arcs. In particular, $c_1$ and $c_n$ are connected by the diameter, and we denote this circular polygon as $P_1$. Reflect $P_1$ with respect to the side $c_1c_n$, it will result in $(n-2)$ points from the reflections of $\{c_2,\ldots,c_{n-1}\}$ on the lower half circle, we mark those points counter-clockwisely by $c_{n+1},\ldots,c_{2n-2}$, and mark the reflected polygon by $P_2$. Then we call the combined polygon $P=P_1+P_2$ a fundamental polygon of this covering (see Figure \ref{fig:PolygoOnDisk}).

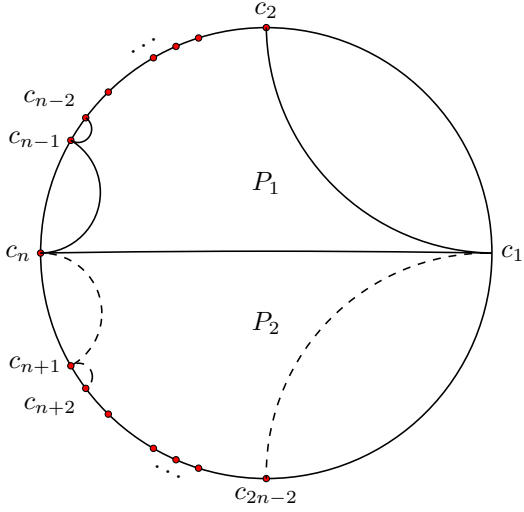
\begin{figure}
	\begin{center}
	\begin{tikzpicture}[scale=3]
	\tkzDefPoint(0,0){O}
	\tkzDefPoint(1,0){A}
	\tkzDrawCircle[D](O,A)
	\tkzDefPoint(-1,0){c1} 
	\tkzDefPoint({-root(2,0.75)},0.5){c2}
	\tkzDefPoint({-root(2,0.75)},-0.5){c2'}
	\tkzDefPoint(-0.8,0.6){c3}
	\tkzDefPoint(-0.8,-0.6){c3'}
	\tkzDefPoint(-0.7,{root(2,0.51)}){c4}
	\tkzDefPoint(-0.7,{-root(2,0.51)}){c4'}
	\tkzDefPoint(-0.5,{root(2,0.75)}){c5}
	\tkzDefPoint(-0.5,{-root(2,0.75)}){c5'}
	\tkzDefPoint(-0.4,{root(2,0.84)}){c6}
	\tkzDefPoint(-0.4,{-root(2,0.84)}){c6'}
	\tkzDefPoint(-0.3,{root(2,0.91)}){c7}
	\tkzDefPoint(-0.3,{-root(2,0.91)}){c7'}
	\tkzDefPoint(0,1){cn}
	\tkzDefPoint(0,-1){cn'}
	\tkzDefPoint(1,0){c0}
	\tkzDrawPoints[color=black,fill=red,size=6](c1,c2,c3,c4,c2',c3',c4',cn,cn',c5,c6,c7,c5',c6',c7')
	\tkzLabelPoint[left](c1){$c_n$}
	\tkzLabelPoint[left](c2){$c_{n-1}$}
	\tkzLabelPoint[left](c2'){$c_{n+1}$}
	\tkzLabelPoint[above left](c3){$c_{n-2}$}
	\tkzLabelPoint[below left](c3'){$c_{n+2}$}
	\tkzLabelPoint[above](cn){$c_2$}
	\tkzLabelPoint[below](cn'){$c_{2n-2}$}
	\tkzLabelPoint[right](c0){$c_1$}
	\tkzLabelPoint[above,rotate=30](c5){$\cdots$}
	\tkzLabelPoint[below,rotate=-25](c6'){$\cdots$}
	
	\tkzDefPoint(0,0.4){P1}
	\tkzLabelPoint[below](P1){$P_1$}
	\tkzDefPoint(0,-0.4){P2}
	\tkzLabelPoint[above](P2){$P_2$}
	\tkzClipCircle(O,A)    
	\tkzDrawCircle[orthogonal through=c1 and c2](O,A)
	\tkzDrawCircle[orthogonal through=c2 and c3](O,A)
	\tkzDrawCircle[orthogonal through=cn and c0](O,A)
	\tkzDrawCircle[orthogonal through=c1 and c0](O,A)
	\tkzDrawCircle[orthogonal through=c1 and c2',dashed](O,A)
	\tkzDrawCircle[orthogonal through=c2' and c3',dashed](O,A)
	\tkzDrawCircle[orthogonal through=cn' and c0,dashed](O,A)
	\end{tikzpicture}
\end{center}
	\caption{Polygon on Disk}
	\label{fig:PolygoOnDisk}
\end{figure}

Next we construct the reproduction of the covering process. Let $V$ be the set of vertexes with either only odd index or even index. For a vertex $c_m\in V$, reflect $P$ with respect to an adjacent side $c_{m-1}c_m$, it results in a new polygon $\tilde{P}=\tilde{P}_1+\tilde{P}_2$, where $\tilde{P}_i$ is the reflecting image of $P_i$, $i=1,2$. Notice that $\tilde{P}_1$ is the image of $P_1$ from reflections of odd times. $\tilde{P}_2$ is the image of $P_1$ from reflections of even times since $\tilde{P}_2$ is the reflection of $P_2$ and $P_2$ is the reflection of $P_1$. Next we reflect $P$ with respect to $c_mc_{m+1}$, which is the other adjacent side of $c_m$. We replicate such reflections on both sides of each vertex in the set $V$, it will result in a new polygon $P^{(1)}$ on $\mathbb{D}$, where $P^{(1)}$ is composed by $2\times(n-1)+1=2n-1$ copies of $P^{(0)}=P$, and it has $2\times(2n-3)\times(n-1)$ sides and vertexes. Keep repeating this process on every $P^{(i)}$, it will cover the unit disk $\mathbb{D}$ as $i$ approaches $\infty$.

If we change the universal covering space to the upper half plane $\mathbb{H}$, the construction will be the same through the Cayley transformation,
\begin{equation}\label{eq:cayley-transformation}
   t:\mathbb{H}\rightarrow\mathbb{D},\qquad z\mapsto t(z)=\frac{z-i}{z+i}.
\end{equation}
For example, Figure \ref{fig:PolygonOnH} shows the corresponding fundamental polygon on $\mathbb{H}$.

\begin{figure}

\begin{center}
	\begin{tikzpicture}[scale=1]
	\draw [->] (-5,0) -- (5,0);
	\draw [->] (0,0) -- (0,4);
	\node [below] at (5,0) {$x$};
	\node [left] at (0,4) {$y$};
	\node at (0,5) {$t^{-1}(c_1)=\infty$};
	\tkzDefPoint(0,0){a1};
	\tkzDefPoint(1.2,0){a2};
	\tkzDefPoint(-1.2,0){a2'};
	\tkzDefPoint(1.7,0){a3};
	\tkzDefPoint(-1.7,0){a3'};
	\tkzDefPoint(2,0){a4};
	\tkzDefPoint(-2,0){a4'};
	\tkzDefPoint(2.2,0){a5};
	\tkzDefPoint(-2.2,0){a5'};
	\tkzDefPoint(3,0){an};
	\tkzDefPoint(-3,0){an'};
	\tkzDrawPoints[color=black,fill=red,size=6](a1,a2,a2',a3,a3',a4,a4',a5,a5',an,an');
	\node [below,scale=0.7] at (a1) {$t^{-1}(c_n)$};
	\node [below,scale=0.7] at (a2) {$t^{-1}(c_{n-1})$};
	\node [below,scale=0.7] at (a2') {$t^{-1}(c_{n+1})$};
	\node [below] at (a4) {$\cdots$};
	\node [below,scale=0.7] at (an) {$t^{-1}(c_2)$};
	\node [below] at (a4') {$\cdots$};
	\node [below,scale=0.7] at (an') {$t^{-1}(c_{2n-2})$};
	\draw (a2) arc [radius=0.6,start angle=0,end angle=180];
	\draw (a3) arc [radius=0.25,start angle=0,end angle=180];
	\draw (an) arc [radius=0.4,start angle=0,end angle=180];
	\node [above,scale=0.7] at (a4) {$\cdots$};
	\draw (an) -- (3,4);
	\node [above] at (1.5,1.5) {$t^{-1}(P_1)$};
	\draw [dashed] (a1) arc [radius=0.6,start angle=0,end angle=180];
	\draw [dashed] (a2') arc [radius=0.25,start angle=0,end angle=180];
	\draw [dashed] (a5') arc [radius=0.4,start angle=0,end angle=180];
	\draw [dashed] (an') -- (-3,4);
	\node [above] at (-1.5,1.5) {$t^{-1}(P_2)$};
	\end{tikzpicture}
\end{center}
	\caption{Polygon on $\mathbb{H}$}
	\label{fig:PolygonOnH}
\end{figure}
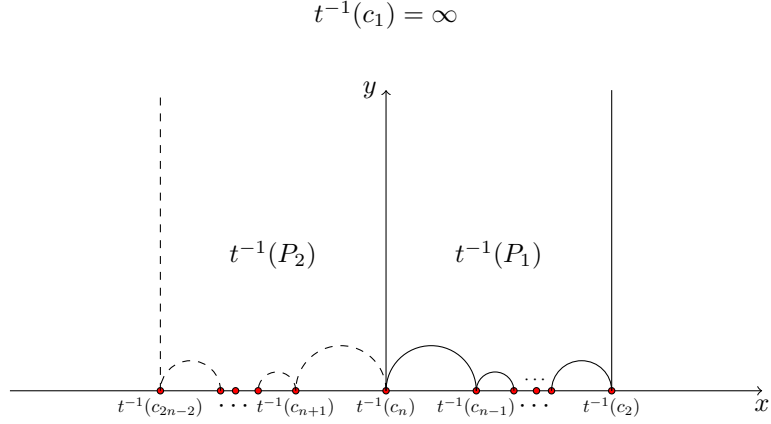
    \section{Deck Transformation Group and Expansion in $q_{_k}$}
\subsection{Deck Transformation Group and Generators}\label{subsection:deck-group-generators}
Let $x\in\punctured$ be a base point and $\gamma:[0,1]\rightarrow\punctured$ be a non-trivial loop with $\gamma(0)=\gamma(1)=x$. Assume $f:\mathbb{H}\rightarrow\punctured$ is a universal covering map. For an arbitrary pre-image $\tau_1\in f^{-1}(x)\subseteq\mathbb{H}$, the end point $\tau_2$ of the lift $\tilde{\gamma}$ with initial point $\tau_1$ is uniquely determined by $\gamma$ and the choice of $\tau_1$. The uniqueness implies that the non-trivial loop $\gamma$ induces a permutation on points of $f^{-1}(x)$, and such permutation induces an isomorphism of its covering space $\mathbb{H}$.
\begin{definition}
	We use $\Aut(f)$ to denote the deck transformation group of the universal covering $f:\mathbb{H}\rightarrow\punctured$. Equivalently, we have $f\circ \gamma(\tau)=f(\tau)$ for every element $\gamma\in\Aut(f)\subseteq\Aut(\mathbb{H})$ and every point $\tau\in\mathbb{H}$.
\end{definition}
There is a one-to-one correspondence between the deck transformation group $\Aut(f)$ and its fundamental group:
$$\Aut(f)\,\leftrightarrow\,\pi_1(\punctured), \qquad h\,\leftrightarrow\,\mbox{loop}.$$
Notice that the fundamental group $\pi_1(\punctured)$ is a free group generated by the $(n-1)$ loops, each of which circles around only one punctured point $a\in\{a_1,\cdots,a_n\}$. Therefore the corresponding relation implies that $\Aut(f)$ is generated by $(n-1)$ elements from $\Aut(\mathbb{H})$. We state the conclusion as the following statement.
\begin{theorem}\label{thm:fundalmental-group-first-thm}
	The deck transformation group $\Aut(f)$ of the universal covering $f:\mathbb{H}\rightarrow\punctured$ is a free group generated by $(n-1)$ elements from $\Aut(\mathbb{H})$, where each generator corresponds to a single loop in $\pi_1(\punctured)$.
\end{theorem} 

Next we will discuss the properties of $\Aut(f)$. Let us recall some elementary definitions.

\begin{definition}\label{def:matrix-acts-on-upper}
	An element $\gamma=\left(\begin{array}{cc}
	a & b\\
	c & d
	\end{array}\right)\in\mbox{GL}_2(\mathbb{C})$ acts on $[z_0:z_1]\in\mathbb{CP}^1$ in the following sense
	\begin{equation}
	\gamma([z_0:z_1])=\left(\begin{array}{cc}
	a & b\\
	c & d
	\end{array}\right)\left(\begin{array}{c}
	z_0\\
	z_1
	\end{array}\right)=\left(\begin{array}{c}
	az_0+b\\
	cz_1+d
	\end{array}\right)=[az_0+b:cz_1+d].\notag
	\end{equation}
\end{definition}

\begin{remark}
	Let $\gamma,\gamma_1,\gamma_2\in\mbox{GL}_2(\mathbb{C})$, and $[z_0:z_1]\in\mathbb{CP}^1$.
	\begin{enumerate}
		\item $\gamma_2(\gamma_1([z_0:z_1]))=(\gamma_2\cdot\gamma_1)([z_0:z_1])$ because of the associativity of matrix multiplication.
		\item If we consider $z=\frac{z_0}{z_1}$ when $z_1\neq 0$, and $z=\infty$ when $z_1=0$, then we have $\gamma(z)=\frac{az+b}{cz+d}$ and $\gamma(\infty)=\frac{a}{c}$.
	\end{enumerate}
\end{remark}

\begin{definition}
	Let $z\in\mathbb{CP}^1$ and an element $\gamma\in\mbox{GL}_2(\mathbb{C})$ such that $\gamma(z)=z$, we say that $z$ is fixed under the action of $\gamma$, or, equivalently, $z$ is invariant under $\gamma$. 
\end{definition}

Elements in $\Aut(\mathbb{H})$ are the ones that we are interested in. On one side, we know that $\Aut(\mathbb{CP}^1)\subseteq\mbox{SL}_2(\mathbb{C})$, therefore we have  $\Aut(\mathbb{H})\subseteq\mbox{SL}_2(\mathbb{C})$. On the other side, the isomorphism of $\mathbb{H}$ keeps its boundary $\mathbb{R}$ invariant, this implies that $\Aut(\mathbb{H})=\mbox{SL}_2(\mathbb{R})$. We will focus on the group $\mbox{SL}_2(\mathbb{R})$ from now on.

\begin{prop}\label{prop:invariant-pt-infty-0}
	Let $\omega$ be a linear transformation in $\mbox{SL}_2(\mathbb{R})$, and $\tau\in\mathbb{H}$.
	\begin{enumerate}
		\item If $0,\infty$ are invariant under the action of $\omega$, i.e., $\omega(0)=0$ and $\omega(\infty)=\infty$, then $\omega(\tau)=\lambda \tau$ for some $0<\lambda\in\mathbb{R}$.
		\item If $\omega$ has only one invariant point $\infty$, i.e., $\omega(\infty)=\infty$, then $\omega(\tau)=\tau+k$ for some $k\in\mathbb{R}$.
		\item If $\omega$ has only one invariant point $i=\sqrt{-1}$, i.e., $\omega(i)=i$, then $\frac{\omega(\tau)-i}{\omega(\tau)+i}=e^{i\theta}\frac{\tau-i}{\tau+i}$ for some $\theta\in\mathbb{R}$.
	\end{enumerate}
\end{prop}
\begin{proof}
	Assume $\omega=\left(\begin{array}{cc}
	a & b\\
	c & d
	\end{array}\right)\in\mbox{SL}_2(\mathbb{R})$. The condition $\omega(\infty)=\infty$ implies that $\frac{a}{c}=\infty$, so we get $c=0$ and $a\neq 0$ since $ad-bc=1$. And $\omega(0)=0$ implies $\frac{b}{d}=0$, so $b=0$ and $d\neq 0$. 
	\begin{enumerate}
      \item If both of $0$ and $\infty$ are invariant under $\omega$, we have $\omega=\left(\begin{array}{cc}
                 a & 0\\
                 0 & d
                 \end{array}\right)$ and $\omega(z)=\frac{a}{d}z$, specially $\lambda=\frac{a}{d}>0$ since $ad=1>0$.\\
      \item If $\infty$ is the only invariant point under the action of $\omega$, we have $\omega=\left(\begin{array}{cc}
                  a & b\\
                  0 & d
               \end{array}\right)$ where $b\neq 0$. Any point other than $\infty$ are not invariant means that the equation $\omega(z)=\frac{a}{d}z+\frac{b}{d}=z$ has no solution, so $\frac{a}{d}=1$. Let $\frac{b}{d}=k$, we conclude $\omega(z)=z+k$.\\
      \item For the last case, the composition $t\circ\omega\circ t^{-1}$ is an automorphism of the unit disk $\mathbb{D}$ with $0$ being fixed, where $t$ denotes the Cayley transformation \eqref{eq:cayley-transformation}. Therefore we have relation $t\circ\omega\circ t^{-1}(\tilde{z})=e^{i\theta}\tilde{z}$ for $\tilde{z}\in\mathbb{D}$. Let $\tilde{z}=t(z)$, the following equation
      $$t\circ\omega\circ t^{-1}(\tilde{z})=t\circ\omega\circ t^{-1}\circ t(z)=t\circ\omega(z)=e^{t\theta}t(z)$$
      holds, where the last equality implies $\frac{\omega(\tau)-i}{\omega(\tau)+i}=e^{i\theta}\frac{\tau-i}{\tau+i}$.
	\end{enumerate}
\end{proof}

\begin{prop}\label{prop:Classification-SL2}
    For an element $\omega\in\mbox{SL}_2(\mathbb{C})$, it fits in one of the following three situations.
	\begin{enumerate}
		\item If $\omega$ has two different invariant points on the boundary $\mathbb{R}\cup\{\infty\}$, it is called of \textit{hyperbolic} type. Furthermore, if $r_1<r_2\in\mathbb{R}$ are the two invariant points, then the transformation has expression $$\frac{\omega(z)-r_2}{\omega(z)-r_1}=\lambda\frac{z-r_2}{z-r_1},\quad\mbox{for some }0<\lambda\in\mathbb{R}.$$
		\item If $\omega$ has only one invariant point on the boundary $\mathbb{R}\cup\{\infty\}$, it is called of \textit{parabolic} type. Furthermore, when the invariant point $r\neq\infty$, it has expression 
		$$-\frac{1}{\omega(z)-r}=-\frac{1}{z-r}+k,\quad\mbox{for some } k\in\mathbb{R}.$$
		\item If $\omega$ has only one invariant point $r\in\mathbb{H}$ in the interior of $\mathbb{H}$, it is called of \textit{elliptic} type. It has expression $$\frac{\omega(z)-r}{\omega(z)-\overline{r}}=e^{i\theta}\frac{z-r}{z-\overline{r}},\quad \mbox{for some }\theta\in\mathbb{R}.$$
	\end{enumerate}
\end{prop}
\begin{proof}
	The proof is elementary. It directly follows from Proposition \ref{prop:invariant-pt-infty-0}. Also see \cite[p. 7, Proposition 1.13]{ShimuraIntroToArith} and \cite[p. 8, section 1.6]{NevaROLF}.
\end{proof}

Next, we will show that each generator in Theorem \ref{thm:fundalmental-group-first-thm} has to be of \textit{parabolic} type. For a universal covering space $f:\mathbb{H}\rightarrow\punctured$, it is equivalent to say that $f$ is biholomorphic. Suppose an element $\gamma\in\Aut(f)$ corresponds to a generating loop $l$ in $\pi_1(\punctured)$, where $l:[0,1]\rightarrow\punctured$ and $l$ isolates only one puncture $a\in\{a_1,\ldots,a_n\}$. Fix a point $\tau_0$ in the fiber $f^{-1}(l(0))$, the lift $\tilde{l}(t)$ of $l(t)$ with initial point $\tau_0$ is uniquely determined and so as its terminal point $\tilde{l}(1)=\tau_1$. Similarly,
the lift $\widetilde{l^n}$ of $l^n=l\ast\cdots\ast l$, is uniquely determined by the fixed initial point $\tau_0$, where $\ast$ means $(l\ast l)(t)=\left\{\begin{array}{ll}
	l(2t), & 0\le t<\frac{1}{2},\\
	l(2t-1), & \frac{1}{2}\le t\le 1.
	\end{array}\right.$
  Denote the terminal point $\tau_n=\widetilde{l^n}(1)$, meanwhile $\tau_n\in f^{-1}(l(0))$. By the corresponding relation between $l$ and $\gamma$, we have $\tau_n=(l\ast\cdots\ast l)^{\sim}(1)=\gamma\circ\cdots\circ \gamma(\tau_0)=\gamma^n(\tau_0)$. We will prove that this $\gamma$ can not be of \textit{hyperbolic} type by analyzing its monodormy behavior. Similarly, we can conclude that $\gamma$ can not be of \textit{elliptic} type either. The proofs of the following propositions are elementary, but it is helpful to understand the uniformizing function at singularities. For this purpose, I will include one of the proofs.

\begin{prop}\label{prop:generator-not-hyper}
	The generator is not of \textit{hyperbolic} type.
\end{prop}
	\begin{proof}
	Suppose $l$ is a loop in $\punctured$ which isolates a puncture $a\in\{a_1,\ldots,a_n\}$. $\gamma$ is the corresponding element in $\Aut(f)$, where $f$ is the universal covering map. If $\gamma$ is of \textit{hyperbolic} type, then there exists $r_1<r_2\in\mathbb{R}$ which are the two invariant points of $\gamma$, and let $\lambda\in\mathbb{R}^+$ be the constant in Proposition \ref{prop:Classification-SL2} such that the relation
	$$\frac{\gamma(\tau)-r_2}{\gamma(\tau)-r_1}=\lambda\frac{\tau-r_2}{\tau-r_1},\qquad\tau\in\mathbb{H}$$
	is satisfied. Therefore the following equality 
	$$
	\frac{\tau_{n+1}-r_2}{\tau_{n+1}-r_1}=\frac{\gamma(\tau_n)-r_2}{\gamma(\tau_n)-r_1}=\lambda\frac{\tau_n-r_2}{\tau_n-r_1}$$
	implies relation
	$$\frac{\widetilde{l^{n+1}}(1)-r_2}{\widetilde{l^{n+1}}(1)-r_1}=\lambda\frac{\widetilde{l^n}(1)-r_2}{\widetilde{l^n}(1)-r_1}=\lambda^n\frac{\widetilde{l}(1)-r_2}{\widetilde{l}(1)-r_1},$$	
	which leads to the following equation
	\begin{equation}\label{eq:generator-not-hyper}
	\frac{\tau_n-r_2}{\tau_n-r_1}=\lambda^n\frac{\tau_0-r_2}{\tau_0-r_1}.
	\end{equation}
	On the branch $0<\arg\left(\frac{\tau-r_2}{\tau-r_1}\right)<\pi$, the following calculation
	\begin{align}\label{eq:not-hyperbolic-1}
	\exp\{\frac{2\pi i}{\log\lambda}\cdot\log\frac{\tau_n-r_2}{\tau_n-r_1}\}
	&=\exp\{\frac{2\pi i}{\log\lambda}\cdot\log\left(\lambda^n\cdot\frac{\tau_0-r_2}{\tau_0-r_1}\right)\},\notag\\
	&=\exp\{\frac{2\pi i}{\log\lambda}\cdot n\log\lambda\}\cdot\exp\{\frac{2\pi i}{\log\lambda}\cdot\log\frac{\tau_0-r_2}{\tau_0-r_1}\},\notag\\
	&=\exp\{\frac{2\pi i}{\log\lambda}\cdot\log\frac{\tau_0-r_2}{\tau_0-r_1}\}
	\end{align}
	holds since $\lambda\in\mathbb{R}^+$. Define a function 
	$$\Phi(x)=\exp\left\{\frac{2\pi i}{\log\lambda}\cdot\log\left(\frac{f^{-1}(x)-r_2}{f^{-1}(x)-r_1}\right)\right\},\quad U_l-\{a\}\rightarrow\mathbb{C},$$
	where $U_{l}$ is the open set bounded by $l$ with $a$ inside. Notice that the value defined by the expression $\exp\left\{\frac{2\pi i}{\log\lambda}\cdot\log\left(\frac{f^{-1}\circ l(0)-r_2}{f^{-1}\circ l(0)-r_1}\right)\right\}$ is single valued over the fiber of $l(0)$ due to equation \eqref{eq:not-hyperbolic-1}. Applying the same argument on every point $x\in U_{l}-\{a\}$, we conclude that $\Phi(x)$ is a well-defined single valued function on $U_{l}-\{a\}$ over the branch $0<\arg \frac{f^{-1}(x)-r_2}{f^{-1}(x)-r_1}<\pi$. We estimate $|\Phi(x)|$ in $U_{l}-\{a\}$ as the following,
	\begin{align}
	|\Phi(x)| &=\left|\exp\left\{\frac{2\pi i}{\log\lambda}\cdot\log\left(\frac{f^{-1}(x)-r_2}{f^{-1}(x)-r_1}\right)\right\}\right|,\notag\\
	&=\left|\exp\left\{\frac{2\pi i}{\log\lambda}\right\}\cdot\log\left|\frac{f^{-1}(x)-r_2}{f^{-1}(x)-r_2}\right|\right|\cdot\exp\left\{-\frac{2\pi}{\log\lambda}\cdot\arg\frac{f^{-1}(x)-r_2}{f^{-1}(x)-r_1}\right\},\notag\\
	&=\exp\left\{-\frac{2\pi}{\log\lambda}\cdot\arg\frac{f^{-1}(x)-r_2}{f^{-1}(x)-r_1}\right\}.\label{eq:not-hyperbolic-2}
	\end{align}
	Over the branch $0<\arg \frac{f^{-1}(x)-r_2}{f^{-1}(x)-r_1}<\pi$, equation \eqref{eq:not-hyperbolic-2} implies that $|\Phi(x)|$ is bounded between two positive values $\exp\left\{-\frac{2\pi^2}{\log\lambda}\right\}$ and $1=\exp\{-\frac{2\pi}{\log\lambda}\cdot 0\}$. By Cauchy-Riemann theorem, $a$ is a removable singularity and $\Phi(a)\neq 0$, thus the exponential of $\Phi(x)$ will be regular at $x=a$, as well as  $f^{-1}(x)$. This contradicts with the fact that $a$ is an essential singularity.
	\end{proof}

\begin{prop}\label{prop:generator-not-elliptic}
	The generator is not of \textit{elliptic} type.
\end{prop}
\begin{proof}
It is similar to the proof of Proposition \ref{prop:generator-not-hyper} (see \cite[p. 16-17, section 3.2]{NevaROLF}).
\end{proof}

From Propositions \ref{prop:generator-not-hyper} and \ref{prop:generator-not-elliptic}, we have the following conclusion.

\begin{theorem}\label{thm:generator-are-parabolic}
	The deck transformation group $\Aut(f)$ of a universal covering space $f:\mathbb{H}\rightarrow\punctured$ is generated by $(n-1)$ \textit{parabolic} elements in $\mbox{SL}_2(\mathbb{R})$.
\end{theorem}
    \subsection{The Expansion and Metric Formula}\label{section:qk-expansion-metric-formula}
Suppose $\gamma_1,\ldots,\gamma_{n-1}$ are the $(n-1)$ \textit{parabolic} generators corresponding to the $(n-1)$ generating loops of $\pi_1(\punctured)$. Theorem \ref{thm:generator-are-parabolic} shows that each of them satisfies the following relation
\begin{equation*}
-\frac{1}{\gamma_j(\tau)-r_j}=-\frac{1}{\tau-r_j}+k_j,\qquad \mbox{when $r_j\neq\infty$},
\end{equation*}
or 
\begin{equation*}
\gamma_j(\tau)=\tau+k_j,\qquad\mbox{when $r_j=\infty$,}
\end{equation*}
where $r_j$ and $k_j$ are the constants corresponding to $\gamma_j$, $j=1,\ldots,n-1$, in Proposition \ref{prop:Classification-SL2}. Let $a_j$ be the singularity that corresponds to $\gamma_j$, the function
\begin{equation*}
	\Phi_j(x)=\left\{\begin{array}{ll}
	e^{-\frac{2\pi i}{k_j}\frac{1}{f^{-1}(x)-r_j}}, & \mbox{when }r_j\neq\infty,\\
	e^{\frac{2\pi i}{k_j}f^{-1}(x)}, & \mbox{when }r_j=\infty,
	\end{array}\right.\qquad U_l-\{a_j\}\rightarrow\punctured
\end{equation*}
is biholomorphic, where $U_l-\{a_j\}$ is a punctured neighborhood that is defined similar to the one in Proposition \ref{prop:generator-not-hyper}. Such function $\Phi_j(x)$ is single valued on the neighborhood $U_l-\{a_j\}$ around the singularity $a_j$. We call this function a uniformizing function at $a_j$. 

Without loss of generality, we will focus on the situation $\infty\in\{r_1,\ldots,r_{n-1}\}$ for convenience. Let $f:\mathbb{H}\rightarrow\punctured$ be a universal covering map, assume $r_1=\infty$, and let $\gamma_1$ be the corresponding generator and $a_1$ be the corresponding singularity. The uniformizing function $\Phi_1(x)$ is defined as the following
\begin{equation*}
\Phi_1(x)=e^{\frac{2\pi i}{k_1}f^{-1}(x)}, \qquad U_l-\{a_1\}\rightarrow\punctured.
\end{equation*}
 If $\gamma_1\in\mbox{SL}_2(\mathbb{R})$ is a generator of $\Aut(f)$ that corresponds to $r_1$, so $\gamma_1$ fixes $\infty$, we have the following periodic property,
\begin{equation*}
f(x)=f\circ\gamma_1(x)=f(x+k_1).
\end{equation*} 
Fundamental Fourier analysis implies the following proposition. 
\begin{prop}\label{prop:coveringmap-global-expression}
        Let  $f$ be a covering map $f:\mathbb{H}\rightarrow\punctured$ such that a generator of $\Aut(f)$ fixes $\infty$. Then the map $f$ has the following global expansion
	\begin{equation}\label{eq:f-qk-expansion}
	f(\tau)=f(\qk)=c_0+A\qk+B\qk^2+\sum_{m=3}^{\infty}c_m\qk^m
	\end{equation}
	in $\qk=\qk(\tau)=\exp\{\frac{2\pi i}{k}\tau\}$ for any $\tau\in\mathbb{H}$, where $k\in\mathbb{R}$ is the constant that corresponds to the generator in Proposition \ref{prop:Classification-SL2}, and the constant $A\neq 0$.
\end{prop}
\begin{proof}
	If $\infty$ is invariant under a generator $\gamma\in\Aut(f)$, then there is a constant $k\in\mathbb{R}$ such that $\gamma(\tau)=\tau+k$, which implies that $f$ has the following periodic property
	$$f(\tau+k)=f(\tau),\qquad \mbox{for any }\tau\in\mathbb{H}.$$
	Consider $f(\tau)=f(x+iy)=f_y(x)$ as a function of $x$ on $\mathbb{R}$. Its Fourier series
	\begin{equation}\label{eq:FourierExpansion1}
	f(x+iy)=f_y(x)=\sum_{n=-\infty}^{\infty}c_n(y)e^{\frac{2\pi}{k}inx},
	\end{equation} 
	converges uniformly on $x$ due to Corollaries 2.3 and 2.4 in \cite[p. 41-42]{SteinFourier}, where the coefficients are given by the following equation
	\begin{equation}
	c_n(y)=\frac{1}{k}\int_{0}^{k}f_y(x)e^{-\frac{2\pi}{k}inx}dx\notag.
	\end{equation}
	The Cauchy-Riemann equation implies the following equation
	\begin{equation}
	i\frac{\partial f}{\partial x}=\frac{\partial f}{\partial y}
	\end{equation}
	holds since $f(\tau)$ is holomorphic. Calculating the left hand side, due to the uniform convergence on $x$, we have equation
	\begin{align}
	i\frac{\partial f}{\partial x}
	&=i\frac{\partial}{\partial x}\left(\sum_{n=-\infty}^{\infty}c_n(y)e^{\frac{2\pi}{k}inx}\right)=i\left(\sum_{n=-\infty}^{\infty}c_n(y)e^{\frac{2\pi}{k_1}inx}\cdot\frac{2\pi}{k}in\right),\notag\\
	&=\sum_{n=-\infty}^{\infty}-\frac{2n\pi}{k}c_n(y)e^{\frac{2\pi}{k}inx}.\notag
	\end{align}
	Also calculating the right hand side, we have equation
	\begin{align*}
	\frac{\partial f}{\partial y}=\sum_{n=-\infty}^{\infty}c_n'(y)e^{\frac{2\pi}{k}inx}.
	\end{align*}
	Matching the coefficients, the differential equation
    \begin{equation*}
	\frac{2n\pi}{k}c_n(y)+c_n'(y)=0
	\end{equation*}
	holds, which has solution
	\begin{equation*}
		c_n(y)=C_ne^{-\frac{2\pi}{k}ny},\quad\mbox{for some constant $C_n\in\mathbb{C}$.}
	\end{equation*}
	Substituting the solution set in equation \eqref{eq:FourierExpansion1}, we have equation
	\begin{equation*}
	f(\tau)=f(x+iy)=\sum_{n=-\infty}^{+\infty}C_ne^{-\frac{2\pi}{k}ny}e^{\frac{2\pi}{k}inx}=\sum_{n=-\infty}^{+\infty}C_ne^{\frac{2\pi}{k}in\tau},
	\end{equation*}
	the last equality holds because $inx-ny=in(x+iy)=in\tau$. On the other hand, the uniformizing function $\qk=e^{\frac{2\pi i}{k}\tau}$ is a local coordinate on a neighborhood of the corresponding singularity $a\in\{a_1,\ldots,a_n\}$ on $\punctured$, so $f$ has a Taylor expansion
	\begin{equation}\label{eq:Taylor-Fourier}
	f(\tau)=c_0+c_1\qk+c_2\qk^2+\sum_{m=3}^{\infty}c_m\qk^m
	\end{equation}
	 in $\qk$ on a neighborhood of $a$, which will coincide with its Fourier series, thus equation \eqref{eq:Taylor-Fourier} holds globally. For convenience in later sections, we write $A=c_1, B=c_2$,
	\begin{equation*}
	f(\tau)=c_0+A\qk+B\qk^2+\sum_{j=3}^{\infty}c_j\qk^j.
	\end{equation*}
	To see $A\neq 0$, remember that $f(\tau)$ is a covering map without ramification point, this implies that $f_{\qk}|_{\qk=0}\neq 0$, so $A\neq 0$.
\end{proof}

The condition $A\neq0$ in equation \eqref{eq:f-qk-expansion} allows us to find an inversion series $\qk(f)$ in $f$ such that $\qk(f(\qk))=\qk$. We assume the inversion series $\qk(f)$ has expression
\begin{equation}\label{eq:qk-f-inversion}
\qk(f)=\tilde{c}_0+\tilde{A}f+\tilde{B}f^2+\sum_{j=3}^{\infty}\tilde{c}_jf^j.
\end{equation}

\begin{theorem}\label{thm:metric-expression}
	Let $f:\mathbb{H}\rightarrow\punctured$ is a covering map for the universal covering space, and $\Aut(f)$ has a generator $\gamma$ which fixes infinity, then there exists a positive constant $k\in\mathbb{R}$ such that $f$ can be given as the following expansion
	\begin{equation*}
	f(\tau)=f(\qk)=c_0+A\qk+B\qk^2+c_3\qk^3+\sum_{m=4}^{\infty}c_m\qk^m
	\end{equation*}
	in $\qk=\exp\{\frac{2\pi i}{k}\tau\}, \tau\in\mathbb{H}$ and $A\neq 0$. Furthermore, such covering map induces a complete K\"{a}hler-Einstein metric on $\punctured$ from $\mathbb{H}$, it can be given by the following equation
	\begin{equation}
	ds^2=\frac{|\qk'(f)|^2}{|\qk(f)|^2\log^2|\qk(f)|}|df|^2,\quad f\in\punctured,
	\end{equation}
	where $\qk(f)$ is the inversion series \eqref{eq:qk-f-inversion}.
\end{theorem}
\begin{proof}
	The first statement is directly from Proposition \ref{prop:coveringmap-global-expression}, we only need to show the second part. For any point $x_0\in\punctured$, there is an evenly covered neighborhood $U$ of $x_0$ that is simply connected. Let $\tilde{U}$ and $\tilde{U}'$ are two different pre-images of $U$, i.e., $\tilde{U}\neq\tilde{U}'\in f^{-1}(U)$. Notice that $\tilde{U}$ and $\tilde{U}'$ are both in $\mathbb{H}$, we can define logarithms of $\exp\{\frac{2\pi i}{k}\tau\}=\qk=\qk(f)$ on different branches as the following
	\begin{align}
	\tau=\tau(f)=\frac{k}{2\pi i}\log\qk(f),\qquad\mbox{on }\tilde{U},\notag\\
	\tau'=\tau'(f)=\frac{k}{2\pi i}\log\qk(f),\qquad\mbox{on }\tilde{U}'.\notag
	\end{align}
	Recall that $\tilde{U}$ and $\tilde{U}'$ are homeomorphic through a cover transformation $h\in\Aut(f)\subseteq\mbox{SL}_2(\mathbb{R})$, i.e., $h\circ \tau(f)=\tau'(f)$. Since the Poincar\'{e} metric is invariant under actions of $\mbox{GL}_2(\mathbb{R})$, therefore we have equality
	\begin{equation}\label{eq:invariance}
	ds^2=\frac{-4}{(\tau-\overline{\tau})^2}|d\tau|^2=\frac{-4}{(\tau'-\overline{\tau'})^2}|d\tau'|^2,
	\end{equation}
	Thus the general logarithm property
	$$\frac{d}{df}(\log \qk(f))=\qk'(f)/\qk(f)df$$
	holds since $\log \qk(f)$ is biholomorphic on the branch $\tilde{U}$, . The metric can be induced on $\tilde{U}$ by the following calculation 
	\begin{align}
	ds^2 & =\frac{-4}{(\tau-\overline{\tau})^2}|d\tau|^2,\notag\\
	& =\frac{-4}{((\frac{k}{2\pi i}\log \qk(f)-\overline{\frac{k}{2\pi i}\log \qk(f)})^2}|d(\frac{k}{2\pi i}\log \qk(f))|^2,\notag\\
	& =\frac{-4\cdot (2\pi i)^2}{k^2(\log \qk(f)+\overline{\log \qk(f)})^2}(\frac{k}{2\pi})^2\left|\frac{\qk'(f)}{\qk(f)}df\right|^2,\notag\\
	& =\frac{|\qk'(f)|^2}{|\qk(f)|^2\log^2|\qk(f)|}|df|^2.\label{eq:metricexpression}
	\end{align}
    Furthermore, the invariant property in equation \eqref{eq:invariance} implies that the expression \eqref{eq:metricexpression} is independent of the choice of branches. Therefore the complete K\"{a}hler-Einstein metric $ds^2$ can be defined globally by equation \eqref{eq:metricexpression} on $\punctured$.
\end{proof}

One thing that is worth to mention is that Theorems \ref{thm:metric-expression} and \ref{thm:coefficient-depend-on-A-B}, which will be included in the later section, together imply the main result Theorem \ref{thm:into-main-1}.
    \section{Schwarzian Derivative}\label{section:schwarzian-derivative}
\subsection{Introduction}
The Schwarzian derivative is invariant under linear transformations, let $\tau=\tau(f)=f^{-1}$ be the inverse of the covering map $f:\mathbb{H}\rightarrow\punctured$, the Schwarzian derivative $\{\tau,f\}$ is actually well-defined due to the invariance. This fact builds a connection between $\{f,\tau\}$ and $\{\tau,f\}$, which will lead to the main result of this article.
\begin{definition}\label{def:schwarzian-deirvative}
	If a function $f$ is locally biholomorphically defined on the complex $z-$plane, the Schwarzian derivative of $f$ with respect to $z$ is defined as the following:
	\begin{equation*}
		\{f,z\}=2\left(\frac{f_{zz}}{f_z}\right)_z-\left(\frac{f_{zz}}{f_z}\right)^2,
	\end{equation*}
	where we write $f_z=df/dz$ for convenience.
\end{definition}

In Definition \ref{def:schwarzian-deirvative}, we require $f$ being locally biholomorphic so that we are able to talk about the Schwarzian derivative of its inverse function $z=f^{-1}$. We simply write $\{z,f\}$ to denote the Schwarzian derivative of $z=z(f)$ with respect to $f$. If $w$ is also a locally biholomorphic function defined on the complex plane, then it allows us to consider the Schwarzian derivative of the composition $f\circ w(z)$. Since the functions $f,w$ are biholomorphic, we will simply write the inverse $f^{-1}$ as $z(f)$ and the inverse $w^{-1}=z(w)$.

\begin{prop}\label{prop:schwarzian-basic-properties}
	Assume $f,w$ are two locally biholomorphic functions defined on the complex complex plane, then we have the following properties:
	\begin{enumerate}
		\item $\{\frac{af+b}{cf+d},z\}=\{f,z\}$ for any $\left(\begin{array}{cc}
		a & b\\
		c & d
		\end{array}\right)\in\mbox{SL}_2(\mathbb{C})$,\\
		\item $\{f,z\}=\{f,w\}w_z^2+\{w,z\}$ if $f=f\circ w(z)$. In particular, $\{z,f\}=-\{f,z\}z_f^2$.
	\end{enumerate}
\end{prop}
\begin{proof}
	The proof is obvious.
\end{proof}

Recall Theorem \ref{thm:metric-expression} in section \ref{section:qk-expansion-metric-formula}, the covering map $f$ can be expressed as an expansion in $\qk$, let us denote this expansion as $f(\qk)$. Notice that the three functions $f(\tau)$ and $f(\qk)$ and $\qk(\tau)$ are all locally biholomorphic, it makes sense to consider the relation of the Schwarzian derivatives among the three of them.

\begin{prop}\label{prop:schwarzian-expression-qk-tau-relation}
	Let $f:\mathbb{H}\rightarrow\punctured, \tau\mapsto f(\tau)$ be a covering map and $\qk(\tau)=\exp\{\frac{2\pi i}{k}\tau\}$, then we have the following relation
	\begin{equation}\label{eq:schwarzian-expression-qk-tau-relation}
		\{f,\tau\}=\frac{4\pi^2}{k^2}(1-\qk^2\{f,\qk\}).
	\end{equation}
\end{prop}
\begin{proof}
	Assume $w=\qk=e^{\frac{2\pi}{k}i\tau}$ in Proposition \ref{prop:schwarzian-basic-properties}, we have equalities
	$$(\qk)_{\tau}=\frac{2\pi i}{k}\qk,\qquad(\qk)_{\tau\tau}=-(\frac{2\pi}{k})^2\qk,\qquad (\qk)_{\tau}^2=-(\frac{2\pi}{k})^2\qk^2,$$
	and
	$$\{\qk,\tau\}=0-(\frac{2\pi i}{k})^2=\frac{4\pi^2}{k^2}.$$
    Direct calculation leads to the conclusion of this proposition,
	\begin{equation}
	\{f,\tau\}=-(\frac{2\pi}{k})^2\qk^2\{f,\qk\}+(\frac{2\pi}{k})^2=\frac{4\pi^2}{k^2}(1-\qk^2\{f,\qk\}).\notag
	\end{equation}
\end{proof}

It is beneficial to have a closer look at the expansion of $\{f,\qk\}$ in $\qk$.

\begin{prop}\label{prop:schwarzian-f-qk-expansion}
    If a covering map $f:\mathbb{H}\rightarrow\punctured$ has the following expansion
    \begin{equation}\label{eq:f-qk-expanion-with-0}
       f(\tau)=f(\qk)=A\qk+B\qk^2+c_3\qk^3+c_4\qk^4+\sum_{m=5}^{\infty}c_m\qk^m
    \end{equation}
    in $\qk=\exp\{\frac{2\pi i}{k}\tau\}$, where $\tau\in\mathbb{H}$. Then the Schwarzian derivative of $f(\qk)$ with respect to $\qk$ has the following expansion
    \begin{equation}\label{eq:schwarzian-f-qk-expansion-2}
    \{f,\qk\}=P_0(\B,\C_3)+\sum_{m=1}^{\infty}P_m(\B,\C_3,\ldots,\C_{m+3})\qk^m
    \end{equation}
     in $\qk$, where $\B=\frac{B}{A}$, $\C_{m+3}=\frac{c_{m+3}}{A}$ and the coefficient term $P_m(\B,\C_3,\ldots,\C_{m+3})$ is a polynomial in $\B,\C_3,\ldots,\C_{m+3}$ with degree $1$ in $\C_{m+3}$, and degree $m+2$ in $\B$, $m\ge 0$. Specially $P_0(\B,\C_3)=12(\C_3-\B^2)$.
\end{prop}
\begin{proof}
	We will write $f'(\qk)=\frac{\d f(\qk)}{\d\qk}$ and $f''(\qk)=\frac{\d}{\d\qk}\left(\frac{\d f}{\d\qk}\right)$ for convenience. Direct calculation gives the following equations
\begin{align}
	f'(\qk)&=A+2B\qk+3c_3\qk^2+\sum_{m=4}^{\infty}mc_m\qk^{m-1}=A+A\cdot\left(2\B\qk+\sum_{m=2}^{\infty}Q^{(1)}_m(\C_{m+1})\qk^m\right),\notag\\
	f''(\qk)&=2B+6c_3\qk+\sum_{m=4}^{\infty}m(m-1)c_m\qk^{m-2}=A\left(2\B+\sum_{m=1}^{\infty}Q^{(2)}_m(\C_{m+2})\qk^m\right),\notag
\end{align}
where $Q^{(1)}_m(\C_{m+1})$ is a polynomial in $\C_{m+1}$ with degree $1$, and $Q^{(2)}_m(\C_{m+2})$ is a polynomial in $\C_{m+2}$ with degree $1$ as well. We have the following calculation
\begin{align}
1/f'(\qk)&=\frac{1}{A}\cdot\frac{1}{1+\left(2\B\qk+\sum_{m=2}^{\infty}Q^{(1)}_m(\C_{m+1})\qk^m\right)},\notag\\
         &=\frac{1}{A}\cdot\left[1-2\B\qk+\left(4\B^2-3\C_3\right)\qk^2+\sum_{m=3}^{\infty}Q^{(3)}_m(\B,\C_3,\ldots,\C_{m+1})\qk^m\right],\label{eq:schwarzian-1/fqk}
\end{align}
where the coefficients $Q^{(3)}_m(\B,\C_3,\ldots,\C_{m+1})$ are polynomials in $\B,\C_3,\ldots,\C_{m+1}$, which has degree $1$ in $\C_{m+1}$ with constant coefficient. 
 We applied the expansion $\frac{1}{1-z}=1+z+z^2+\sum_{m=3}^{\infty}z^m$ in the above calculation. The series \eqref{eq:schwarzian-1/fqk} converges since that $f'(\qk)$ is still analytic on the unit disk $\mathbb{D}$ (see \cite[p. 179-182]{AhlforsComplex}). Then we have the following equation 
\begin{align}
f''(\qk)/f'(\qk)=2\B+(6\C_3-4\B^2)\qk+\sum_{m=2}^{\infty}Q^{(4)}_m(\B,\C_3,\ldots,\C_{m+2})\qk^2,\notag                 
\end{align}
where $Q^{(4)}_m$ are polynomials in $\B,\C_3,\ldots,\C_{m+2}$ with degree $1$ in $\C_{m+2}$ with constant coefficient. 
Continue computing the terms in $\{f,\qk\}=2\left(\frac{f''(\qk)}{f'(\qk)}\right)'-\left(\frac{f''(\qk)}{f'(\qk)}\right)^2$, the Schwarzian derivative can be given by the following expansion
\begin{align*}
\{f,\qk\}&=\sum_{m=0}^{\infty}Q^{(5)}_m(\B,\C_3,\ldots,\C_{m+3})\qk^m-\sum_{m=0}^{\infty}Q^{(6)}_m(\B,\C_3,\ldots,\C_{m+2})\qk^m,\notag\\
         &=12(\C_3-\B^2)+\sum_{m=1}^{\infty}P_m(\B,\C_3,\ldots,\C_{m+3})\qk^m
\end{align*}
 in $\qk$, where $Q^{(5)}_m$ is a polynomial in $\B,\C_3,\ldots,\C_{m+3}$ which has degree $1$ in $\C_{m+3}$ with constant coefficient; 
$Q^{(6)}_m$ is a polynomial in $\B,\C_3,\ldots,\C_{m+2}$ that has degree $1$ in $\C_{m+2}$ but with coefficient in terms of $\B$. 
Therefore the coefficient $P_m(\B,\C_3,\ldots,\C_{m+2})=Q^{(5)}_m-Q^{(6)}_m$ of $\qk^m$ is a polynomial in $\B,\C_3,\ldots,\C_{m+2}$, which has degree $1$ in $\C_{m+3}$ with constant coefficient. 
\end{proof}

\subsection{Differential Equation}
Recall section \ref{subsection:deck-group-generators}, the branches of the inverse $\tau(f)$ are related by linear transformations, Proposition \ref{prop:schwarzian-basic-properties} implies that $\{\tau(f),f\}=\{\tau,f\}$ is well-defined. In order to discuss its analytic property, we consider the following uniformizing function near each singularity $a\in\{a_1,\ldots,a_n\}$, 
\begin{equation}
	\Phi(f)=\exp\{-\frac{2\pi i}{k}\frac{1}{\tau(f)-r}\},\quad\mbox{when }r\neq\infty\label{eq:schwarzian-uniformize-1},
\end{equation}
or
\begin{equation}
	\Phi(f)=\exp\{\frac{2\pi i}{k}\tau(f)\},\quad \mbox{when }r=\infty,\label{eq:schwarzian-uniformize-2}
\end{equation}
where $r$ and $k$ are the constants with respect to the parabolic transformation $\gamma$ in Propositions \ref{prop:invariant-pt-infty-0} and \ref{prop:Classification-SL2}, and $\gamma$ is the generator corresponding to the singularity $a$. On the other side, from section \ref{subsection:deck-group-generators}, the uniformizing functions are single valued in a neighborhood of $a$. So the uniformizing function $\Phi(f)$ has the following expansion near $a$,
\begin{equation}
    \Phi(f)=\alpha_a(f-a)+\beta_a(f-a)^2+(f-a)^2\phi(f),\quad a\neq\infty,\alpha_a\neq 0,\label{eq:unif-func-singularity-expansion-finite}\\
\end{equation}
\begin{equation}
    \Phi(f)=\alpha_{\infty} f+\beta_{\infty}+\varphi(\frac{1}{f}),\quad a=\infty,\alpha_{\infty}\neq 0,\label{eq:unif-func-singularity-expansion-infty}
\end{equation}
where $\phi(f)$ and $\varphi(\frac{1}{f})$ are regular at $a$ and satisfying
\begin{align*}
\phi(a)=0,\qquad\mbox{when }a\neq\infty,\\
\varphi(\frac{1}{a})=\varphi(0)=0,\qquad\mbox{when }a=\infty.
\end{align*}
Assume one of the singularities is $0$, and the parabolic transformation $\gamma$ that associates to this singularity $a=0$ fixes $r=\infty$, and $k$ is the constnat in Proposition \ref{prop:Classification-SL2}, then we have the following analytic expansion of the uniformizing function $\Phi(f)$:
\begin{equation}\label{eq:unif-func-singularity-expansion-calculation}
	\exp\{\frac{2\pi i}{k}\tau(f)\}=\alpha_{(a=0)} f+\beta_{(a=0)} f^2+f^2\phi(f).
\end{equation}
We use equation \eqref{eq:unif-func-singularity-expansion-calculation} to calculate its Schwarzian derivative $\{\tau,f\}$. Notice that differentiation will make the branch problem no longer an issue if we take the derivative of its logarithm, we write $\alpha=\alpha_{(a=0)}$ and $\beta=\beta_{(a=0)}$ for convenience,
\begin{align}
	\frac{2\pi i}{k}\tau(f)&=\log\alpha+\log f+\log(1+\frac{\beta}{\alpha}f+\frac{1}{\alpha}f\phi(f)),\notag\\
	\frac{2\pi i}{k}\tau_{f}&=\frac{1}{f}+\frac{\frac{\beta}{\alpha}+\frac{1}{\alpha}(\phi(f)+f\phi'(f))}{1+\frac{\beta}{\alpha}f+\frac{1}{\alpha}f\phi(f)}=\frac{1}{f}+\frac{\beta}{\alpha}+\phi_1(f),\notag\\
	\log{\frac{2\pi i}{k}}+\log\tau_f&=\log\frac{1}{f}+\log\left(1+\frac{\beta}{\alpha}f+f\phi_1(f)\right),\notag\\
	\frac{\tau_{ff}}{\tau_f}&=-\frac{1}{f}+\frac{\frac{\beta}{\alpha}+\phi_2(f)}{1+\frac{\beta}{\alpha}f+f\phi_1(f)}=-\frac{1}{f}+\frac{\beta}{\alpha}+\phi_3(f)\notag,
\end{align}
where each $\phi_l(f)$, $l=1,2,3$, is regular at $a=0$, and $\phi_l(0)=0$. Therefore the Schwarzian derivative has expansion around $a=0$ as the following equation
\begin{align}
\{\tau,f\}&=2\left(\frac{\tau_{ff}}{\tau_f}\right)_f-\left(\frac{\tau_{ff}}{\tau_f}\right)^2,\notag\\
          &=\frac{1}{f^2}+\frac{2\beta}{\alpha}\frac{1}{f}+\phi_4(f),\label{eq:tau-f-schwarzian-expansion-special-case}
\end{align}
where $\phi_4(f)$ is regular at $a=0$. Recall that $\{\tau,f\}$ is invariant under any linear transformation on $\tau$, so we have the following equality
\begin{equation*}
	\{\tau(f),f\}=\{\frac{1}{\tau(f)-r},f\}.
\end{equation*}
For the special case $a=\infty$, equation \eqref{eq:unif-func-singularity-expansion-infty} implies that $\{\tau,f\}$ does not have a singularity at $a=\infty$. Therefore, due to Liouville's theorem, we have the following conclusion from equation \eqref{eq:tau-f-schwarzian-expansion-special-case} (see \cite[p. 201]{NehariConformalMap}),
\begin{equation}
\{\tau,f\}=\left\{
\begin{aligned}
&\sum_{j=1}^{n}\frac{1}{(f-a_j)^2}+\frac{2\beta_j/\alpha_j}{f-a_j}, \qquad\mbox{if }\infty\notin\{a_1,\ldots,a_n\},\\
&\sum_{j=1}^{n-1}\frac{1}{(f-a_j)^2}+\frac{2\beta_j/\alpha_j}{f-a_j}, \qquad\mbox{if }a_n=\infty.
\end{aligned}\right.
\end{equation}
From now on, we will fix one of the singularities to be $0$. Assume $f:\mathbb{H}\rightarrow\mathbb{CP}^1\setminus\{a_1=0,a_2,\ldots,a_n\}, \tau\mapsto f(\tau)$ is a covering map, then we have the following equation
\begin{equation}\label{eq:right-hand-side-of-relation}
	\{\tau,f\}=\frac{1}{f^2}+\frac{2\beta}{\alpha}\frac{1}{f}+\sum_{j=2}^{n}\left[\frac{1}{(f-a_j)^2}+\frac{2\beta_j}{\alpha_j}\frac{1}{f-a_j}\right],
\end{equation}
where $\alpha=\alpha_{(a_1=0)},\beta=\beta_{(a_1=0)}$, and $\alpha_j=\alpha_{a_j},\beta_j=\beta_{a_j}$, $j=2,\ldots,n$, are the coefficients of the first and second order leading terms in the analytic expansion \eqref{eq:schwarzian-uniformize-1} of the unifomizing function at the singularity $a_j$. It is worth to mention that $\alpha_j,\beta_j$ are uniquely determined by the given covering map $f$ due to its analytic property.

Let us recall Propositions \ref{prop:schwarzian-basic-properties} and \ref{prop:schwarzian-expression-qk-tau-relation}, we have the following equation
\begin{equation}
	\{\tau,f\}=-\frac{4\pi^2}{k^2}(1-\qk^2\{f,\qk\})\tau_f^2\notag
\end{equation}
and equation
\begin{equation}
	\{\tau,f\}=\frac{1}{f^2_{\qk}}(\frac{1}{\qk^2}-\{f,\qk\})
\end{equation}
due to $\tau_f=\frac{k}{2\pi i}\qk^{-1}f_{\qk}^{-1}$ from chain rule. Therefore, we have the following relation that is connected by $\{\tau,f\}$,
\begin{equation}\label{eq:schwarzian-qk-relation}
	\frac{1}{f^2_{\qk}}(\frac{1}{\qk^2}-\{f,\qk\})=\frac{1}{f^2}+\frac{2\beta}{\alpha}\frac{1}{f}+\sum_{j=2}^{n}\left[\frac{1}{(f-a_j)^2}+\frac{2\beta_j}{\alpha_j}\frac{1}{f-a_j}\right].
\end{equation}
From equations \eqref{eq:schwarzian-f-qk-expansion-2} and \eqref{eq:schwarzian-1/fqk}, we have the following expansion of the left hand side of equation \eqref{eq:schwarzian-qk-relation},
\begin{equation}
   \frac{1}{A^{2}}\left[\frac{1}{q_{k}^{2}}-4\B\frac{1}{q_{k}}+(24\B^2-18\C_3)+\sum_{m=1}^{\infty}\tilde{P}_m(\B,\C_3,\ldots,\C_{m+3})\qk^m\right],\label{eq:schwarzian-tau-f-qk-expansion} 
\end{equation}
where the coefficient term $\tilde{P}_m(\B,\C_3,\ldots,\C_{m+3})$ is a polynomial in $\B,\C_3,\ldots,\C_{m+3}$ which has degree $1$ in $\C_{m+3}$ for $m\ge1$. 
For convenience, we denote the constant term $\tilde{P}_0(\B,\C_3)=(24\B^2-18\C_3)$.

On the punctured sphere $\mathbb{CP}^1\setminus\{a_1=0,a_2,\ldots,a_n\}$, the singularities are discrete since it is a set of finite points. Take a neighborhood $U$ which contains the one and only singularity $a_1=0$, then we have the following expansion in $U$,
\begin{equation}
\frac{1}{f}=\frac{1}{A}\frac{1}{\qk}[1-\B\qk+\sum_{m=2}^{\infty}Q^{(7)}_m(\B,\C_3,\ldots,\C_{m+1})\qk^m],
\end{equation}
where $Q^{(7)}_m$ is a polynomial in $\B,\C_3,\ldots,\C_{m+1}$ which has degree $1$ in $\C_{m+1}$ with constant coefficient. 
And for other terms $\frac{1}{f-a_j}$, $j=2,\ldots,n$, it has expansion 
\begin{align}
	\frac{1}{f-a_{j}}=-\frac{1}{a_{j}}\cdot\left[1+\frac{A}{a_{j}}\qk+\sum_{m=2}^{\infty}Q^{(8)}_m(\frac{A}{a_j},\frac{c_2}{a_j},\ldots,\frac{c_{m}}{a_j})\qk^m\right],
\end{align}
where $c_2=B$, $Q^{(8)}_m$ is a polynomial in $\frac{A}{a_j},\frac{c_2}{a_j},\ldots,\frac{c_{m}}{a_j}$ which has degree $1$ in $\frac{c_m}{a_j}$ with constant coefficient. 
Therefore equation \eqref{eq:right-hand-side-of-relation} on the right hand side of equation \eqref{eq:schwarzian-qk-relation} has expansion in $\qk$ as the following
\begin{equation}\label{eq:tau-f-right-hand-side-expansion-specific}
\begin{aligned}
\frac{1}{A^2}&\left\{\frac{1}{\qk^2}-4\B\frac{1}{\qk}+\left[5\B^2-2\C_3+A^2\sum_{j=2}^{n}\left(\frac{1}{a^2_j}-\frac{1}{a_j}\frac{2\beta_j}{\alpha_j}\right)\right]\right.\\
&\qquad+\left.\sum_{m=1}^{\infty}\left[Q^{(9)}_m(\B,\C_3,\ldots,\C_{m+3})+\sum_{j=2}^{n}Q^{(10)}_{j_m}(\frac{1}{a_j},c_1,\ldots,c_m)\right]\qk^m\right\}, 
\end{aligned}
\end{equation}
where $c_1=A,c_2=B$, and $Q^{(9)}_m(\B,\C_3,\ldots,\C_{m+3})$ is a polynomial in $\B,\C_3,\ldots,\C_{m+3}$ that has degree $1$ in $\C_{m+3}$ with constant coefficient; 
$Q^{(10)}_{j_m}(\frac{1}{a_j},c_1,\ldots,c_m)$ is a polynomial in $\frac{1}{a_j},c_1,\ldots,c_m$ which has degree $1$ in $c_m$ with constant coefficient. 
If we consider each $c_m=\C_m\cdot A$ for $m\ge 1$, then $Q^{(10)}_{j_m}$ is a polynomial $Q^{(10)'}_{j_m}$ in $A,\B,\C_3,\ldots,\C_m$. We identify equation \eqref{eq:tau-f-right-hand-side-expansion-specific} with equation \eqref{eq:schwarzian-tau-f-qk-expansion} to get a set of equations
\begin{equation}\label{eq:schwarzian-analytic-equation-set}
	\tilde{P}_m(\B,\C_3,\ldots,\C_{m+3})=Q^{(9)}_m(\B,\C_3,\ldots,\C_{m+3})+\sum_{j=2}^{n}Q^{(10)'}_{j_m}(A,\B,\C_3,\ldots,\C_m),
\end{equation}
where $m\ge 1$. And the constant term gives equation
\begin{equation}
	24\B^2-18\C_3=5\B^2-2\C_3+A^2\sum_{j=2}^{n}\left(\frac{1}{a^2_j}-\frac{1}{a_j}\frac{2\beta_j}{\alpha_j}\right),
\end{equation}
thus the solution for $\C_3$ is given by the following equation
\begin{equation}
	\C_3=\C_3(A,\B)=\frac{1}{16}\left[19\B^2-A^2\sum_{j=2}^{n}\left(\frac{1}{a^2_j}-\frac{1}{a_j}\frac{2\beta_j}{\alpha_j}\right)\right].
\end{equation}
Notice that the left hand side of equation \eqref{eq:schwarzian-analytic-equation-set} is a polynomial in $\B,\C_3,\ldots,\C_{m+3}$ that has degree $1$ in $\C_{m+3}$ with constant coefficient, and the right hand side of equation \eqref{eq:schwarzian-analytic-equation-set} is a polynomial in $A,\B,\C_3,\ldots,\C_{m+3}$ that also has degree $1$ in $\C_{m+3}$ with constant coefficient. Let us start with $m=1$, equation \eqref{eq:schwarzian-analytic-equation-set} has one and the only one unknown term $\C_4$, we can solve for $\C_4$ and the solution is a polynomial in $\C_3,A$ and $\B$. The solution $\C_4(\C_3,A,\B)$ can be expressed as a polynomial $\C_4(A,\B)$ in $A,\B$ with constant coefficients since $\C_3=\C_3(A,\B)$. By induction, it is easy to conclude that $\C_m$ can be solved as a polynomial $\C_m(A,\B)$ in $A,\B$ with constant coefficients for $m\ge 3$, i.e.,
\begin{equation}
	\C_m=\C_m(A,\B,a_2,\ldots,a_n)=\C_m(A,\B),
\end{equation}
where $a_2,\ldots,a_n$ are the singularities. Therefore we can conclude the following theorem.
\begin{theorem}\label{thm:coefficient-depend-on-A-B}
	Let $f:\mathbb{H}\rightarrow\mathbb{CP}^1\setminus\{a_1=0,a_2,\ldots,a_n\}$ be a covering map, and the parabolic generator corresponding to $a_1=0$ fixes infinity, then the map $f$ can be uniquely determined up to the coefficients $A,B$ of the first two order leading terms of its expansion \eqref{eq:f-qk-expanion-with-0}, i.e., 
	\begin{equation}\label{eq:coefficient-depend-on-A-B}
		f/A=\qk+\B\qk^2+\C_3(A,\B)\qk^3+\sum_{m=4}^{\infty}\C_m(A,\B)\qk^m,
	\end{equation}
	where $\qk=\exp\{\frac{2\pi i}{k}\tau\}$, and $\C_m(A,\B)$ are polynomials in $A,\B=\frac{B}{A}$ with constant coefficients for $m\ge 3$.
\end{theorem}

Recall Theorem \ref{thm:metric-expression} in section \ref{section:qk-expansion-metric-formula} and Theorem \ref{thm:coefficient-depend-on-A-B}, they together imply our main result Theorem \ref{thm:into-main-1}, which will be concluded in the last section.

    \section{Ramification Point}\label{section:ramification-point}
\subsection{Introduction}
In section \ref{subsection:deck-group-generators}, we concluded that the deck transformation group $\Aut(f)$ is generated by $(n-1)$ \textit{parabolic} transformations from $\Aut(\mathbb{H})$. On the other hand, we will discover properties of such subgroups of $\mbox{GL}_2(\mathbb{R})$. Notice that $\gamma(\tau)=(\frac{\gamma}{\det{\gamma}})(\tau)$, so we will only consider $\mbox{SL}_2(\mathbb{R})$ from now. Consider $\mbox{SL}_2(\mathbb{R})$ with the normal $\mathbb{R}^4$ topology. The general matrix multiplication
$$\mbox{SL}_2(\mathbb{R})\times\mbox{SL}_2(\mathbb{R})\rightarrow\mbox{SL}_2(\mathbb{R}),\quad (\gamma_1,\gamma_2)\mapsto \gamma_1\gamma_2$$
and inversion
$$\mbox{SL}_2(\mathbb{R})\rightarrow\mbox{SL}_2(\mathbb{R}),\quad \gamma\mapsto\gamma^{-1}$$
are continuous maps.
\begin{definition}
	We say that $\mbox{SL}_2(\mathbb{R})$ acts on $\mathbb{H}$ continuously if the following map
	\begin{equation}
	\mbox{SL}_2(\mathbb{R})\times\mathbb{H}\rightarrow \mathbb{H},\qquad (\gamma,\tau)\mapsto\gamma(\tau)\notag
	\end{equation}
	is continuous.
\end{definition}

	In general, if $G$ is a subgroup of $\mbox{SL}_2(\mathbb{R})$. An orbit of a point $\tau\in\mathbb{H}$ is the set of images of $\tau$ under actions in $G$, i.e., 
	$$\mbox{Orbit}_G(\tau)=\{g(\tau)\,|\,g\in G\}.$$
	The stabilizer of a point $\tau\in\mathbb{H}$ are the elements in $G$ that fixes $\tau$, i.e., 
	$$\mbox{Stab}_G(\tau)=\{g\,|\,g(\tau)=\tau, g\in G\}.$$
	We say that two points $\tau,\tau'\in\mathbb{H}$ are equivalent under the action of $G$ if $\tau'\in\mbox{Orbit}_G(\tau)\tau$, i.e.,
	$$\tau\sim\tau'\qquad\mbox{if and only if}\qquad\tau'=g(\tau),\,\mbox{for some } g\in G.$$
	The set $\mbox{Orbit}_G(\tau)$ is the set of points that are equivalent to $\tau$ under the action of $G$.

\begin{definition}
	If $G$ is a discrete subgroup of $\mbox{SL}_2(\mathbb{R})$, define the quotient space $\mathbb{H}/G$ to be the space of equivalent classes of $\mathbb{H}$ under the action of $G$, i.e.,
	$$\mathbb{H}/G=\mathbb{H}/\sim,$$
	equipped the induced topology through the quotient map.
\end{definition}
Let $\mbox{SL}_2(\mathbb{R})/G$ to be the left cosets of $G$ in $\mbox{SL}_2(\mathbb{R})$, i.e.,
$$\mbox{SL}_2(\mathbb{R})/G=\bigcup_{h\in H}hG,$$
where $H$ is the left coset representation  of $G$ in $\mbox{SL}_2(\mathbb{R})$. The topology on $\mbox{SL}_2(\mathbb{R})/G$ is induced from the left group action. 

\begin{definition}\label{def:elliptic-point}
	A point $P\in\mathbb{H}$ is \textit{elliptic} of a discrete subgroup $G\subset\mbox{SL}_2(\mathbb{R})$ if it is invariant under the action of an \textit{elliptic} element in $G$. 
\end{definition}

The goal for this section is to show that a point $P\in\mathbb{H}/G$ is \textit{elliptic} if and only if the point $P\in\mathbb{H}/G$ is a ramification point.

\subsection{Discrete Group, Quotient Space and Ramification Point}
All the conclusions can be found in \cite[p. 2-20, Chapter 1]{ShimuraIntroToArith}. For the completion of this article, I still state them.
\begin{lemma}\label{lemma:fact1-SO2-homeomorphism}
	$\mbox{SO}_2(\mathbb{R})$ is compact in $\mbox{SL}_2(\mathbb{R})$. Furthermore, the map 
	\begin{equation}
	\tilde{\alpha}:\mbox{SL}_2(\mathbb{R})/\mbox{SO}_2(\mathbb{R})\rightarrow\mathbb{H},\qquad [\gamma]\mapsto \tilde{\alpha}([\gamma])=\gamma(i)
	\end{equation}
	is a homeomorphism.
\end{lemma}
\begin{proof}
	See \cite[p. 2, Theorem 1.1]{ShimuraIntroToArith} or \cite[p. 25, Proposition 2.1 (d)]{JSMilne}
\end{proof}

\begin{lemma}\label{lemma:fact2-discrete-compact-finite}
	Assume $\Gamma$ is a discrete subgroup of $\mbox{SL}_2(\mathbb{R})$, if $V_1,V_2$ are any compact subsets in $\mathbb{H}$, then the set $\{\gamma\in\Gamma|\gamma(V_1)\cap V_2\neq\emptyset\}$ is finite.
\end{lemma}
\begin{proof}
	See \cite[p. 3, Proposition 1.6]{ShimuraIntroToArith} or \cite[p. 26, Proposition 2.4]{JSMilne}.
\end{proof}

\begin{prop}\label{prop:discreteneighborhood}
	If $\Gamma\in\mbox{SL}_2(\mathbb{R})$ is a discrete subgroup, then for any $\tau\in\mathbb{H}$, there is a neighborhood $U$ of $\tau$ such that if $\gamma\in\Gamma$ and $U\cap\gamma(U)\neq\emptyset$, then $\gamma(\tau)=\tau$.
\end{prop}
\begin{proof}
	See \cite[p. 3, Proposition 1.7]{ShimuraIntroToArith} or \cite[p. 27, Proposition 2.5 (b)]{JSMilne}
\end{proof}

Next, we are able to show that if $\Gamma$ is a discrete subgroup of $\mbox{SL}_2(\mathbb{R})$, a point $[\tau]=[\tau_0]\in\mathbb{H}/\Gamma$ is ramified if and only if following condition
\begin{equation}\label{eq:elliptic-point-equivalent-condition}
	\mbox{Stab}_{\Gamma}(\tau)-\{I\}\neq \emptyset
\end{equation}
holds, for any $\tau\in[\tau_0]\subset\mathbb{H}$. Recall Definition \ref{def:elliptic-point}, a point $P$ satisfying condition \eqref{eq:elliptic-point-equivalent-condition} if and only if $P$ is an \textit{elliptic} point. The following proposition can be concluded from \cite[p. 8, Proposition 1.18]{ShimuraIntroToArith}, we will present proof here since it is not obvious.

\begin{prop}\label{proposition-5.7}
	Let $\Gamma$ be a discrete subgroup of $\mbox{SL}_2(\mathbb{R})$. Then on the quotient space $\mathbb{H}/\Gamma$, a point $\tau\in\mathbb{H}$ is ramified if and only if it is an \textit{elliptic} point of $\Gamma$.
\end{prop}
\begin{proof}
	If $\tau\in\mathbb{H}$ is not an elliptic point, then by Proposition \ref{prop:discreteneighborhood}, there is a neighborhood $U$ of $\tau$ such that $\gamma(U)\cap U=\emptyset$, for any $\gamma\in\Gamma$. It implies that $\tau$ is not ramified. If $\tau\in\mathbb{H}$ is an elliptic point of $\Gamma$, assume there is an elliptic element $\gamma\in\Gamma$ such that $\gamma(\tau)=\tau$. Let $\sigma\in\mbox{SL}_2(\mathbb{R})$ such that $\sigma(i)=\tau$, then the composition $\sigma^{-1}\circ\gamma\circ\sigma\in\mbox{Stab}(i)=\mbox{SO}_2(\mathbb{R})$, so $<\sigma^{-1}\circ\gamma\circ\sigma>=\sigma^{-1}\circ<\gamma>\circ\sigma$ is a subgroup of the conjugate $\sigma^{-1}\Gamma\sigma$, which is still a discrete subgroup of $\mbox{SL}_2(\mathbb{R})$. Also we have the following correspondence relation
	\begin{equation*}
	<\gamma>\cong \sigma^{-1}<\gamma>\sigma=\mbox{SO}_2(\mathbb{R})\bigcap\sigma^{-1}\Gamma\sigma,
	\end{equation*}
	the right hand side is an intersection of a compact set and a discrete set, so it has to be finite. Therefore if $m=\#\mbox{Stab}_{\Gamma}(\tau)$, $\tau$ is a ramification point of index $m$.
\end{proof}
    \section{Modular Group}\label{section:modular-group}
\subsection{The Full Modular Group $\Gamma(1)$}
In this section, we will start with introducing the full modular group.
\begin{definition}
	The full modular group $\Gamma(1)$ is defined to be the image of $\mbox{SL}_2(\mathbb{Z})$ by identifying $+\gamma$ and $-\gamma$ for any element $\gamma\in\mbox{SL}_2(\mathbb{R})$. Equivalently, it is the same as the following definition:
	\begin{equation}
		\Gamma(1)=\mbox{PSL}_2(\mathbb{Z})=\mbox{SL}_2(\mathbb{Z})/\pm I.\notag
	\end{equation}
\end{definition}
	
In this article, we focus on the principal congruence subgroups of $\Gamma(1)$.

\begin{definition}\label{def:principal-congruence-subgroup}
	The principle congruence subgroup of level $N$ is defined as the following:
	\begin{equation}
		\Gamma(N)=\left\{\left(\begin{array}{cc}
		a & b\\
		c & d
		\end{array}\right)\in\mbox{SL}_2(\mathbb{Z})\left|\left(\begin{array}{cc}
			a & b\\
			c & d
		\end{array}\right)\equiv\pm\left(\begin{array}{cc}
		 1 & 0\\
		0 &  1
	\end{array}\right)\,(\mbox{mod}\,N)\right\}\right. /\pm I.\notag
\end{equation}
\end{definition}

The goal is to determine all the principal congruence subgroups $\Gamma(N)$ such that they are candidates for $\Aut(f)$, i.e., $\mathbb{H}/\Aut(f)=\mathbb{H}/\Gamma(N)\cong\punctured$ for some $n\ge 3$. Notice that $\punctured$ does not have any ramification point, and the compactification of $\punctured$ is the Riemann sphere, which is of genus $zero$. We will use these two properties to determine all suitable $\Gamma(N)$. First, we will use the \textit{fundamental domain} of $\Gamma(1)$ to locate all ramification points of $\mathbb{H}/\Gamma(1)$.
\begin{definition}\label{def:fundamental domain}
	The \textit{fundamental domain} for a discrete subgroup $\Gamma\subseteq\mbox{SL}_2(\mathbb{R})$ is a connected open subset $D$ of $\mathbb{H}$ such that every pair of points in $D$ are inequivalent under $\Gamma$, and meanwhile $\mathbb{H}\subseteq\bigcup_{\gamma\in\Gamma}[\gamma(\overline{D})]$. 
\end{definition}

These conditions are equivalent to $D\rightarrow\mathbb{H}/\Gamma$ is injective and $\overline{D}\rightarrow\mathbb{H}/\Gamma$ is surjective. Now we recall the \textit{fundamental domain} of $\Gamma(1)$ (see \cite[p. 16]{ShimuraIntroToArith}, \cite[p. 19, Proposition 1.2.2]{BumpAuto} or \cite[p. 78-79]{SerreJP}).

\begin{theorem}\label{fact:fundamental-domain-Gamma-1}
	The fundamental domain of $\Gamma(1)$ is the set
	\begin{equation}
	D=\{z=x+yi\,|\,|z|>1,|x|<\frac{1}{2},y>0\}.
	\end{equation}
\end{theorem}

\begin{prop}  
	Let $p$ be the quotient map 
	$$p:\mathbb{H}\rightarrow\mathbb{H}/\Gamma(1),\qquad \tau\mapsto [p(\tau)].$$
	The \textit{elliptic} points in $\mathbb{H}/\Gamma(1)$ are either $[i]$ or $[\rho]$, i.e.,
	$$\{\mbox{ramification points of }\mathbb{H}\rightarrow\mathbb{H}/\Gamma(1)\}=p^{-1}(i)\bigcup p^{-1}(\rho),$$
	where $\rho=e^{\frac{\pi}{3}i}=\frac{1+\sqrt{3}i}{2}$, the root of $1$ of order $6$.
\end{prop}
\begin{proof}
	See \cite[p. 14-15, section 1.4]{ShimuraIntroToArith}.
\end{proof}

\subsection{Genus Formula for Subgroups of Modular Group}
We consider a subgroup $\Gamma$ of $\Gamma(1)$ with finite index, then the quotient space $\mathbb{H}/\Gamma$ is composed by copies of $\mathbb{H}/\Gamma(1)$. The genus of $\mathbb{H}/\Gamma$ can be computed by applying the Riemann-Roch Theorem and the Riemann-Hurwitz Formula.
\begin{prop}\label{prop:genus-formula-subgroup-of-Gamma1}
	If $\Gamma$ is a subgroup of $\Gamma(1)$ with finite index $m$, then the genus $g$ of $\mathbb{H}/\Gamma$ is given by the following formula
	\begin{equation}\label{eq:genus-formula-subgroup-of-Gamma1}
	g=1+m/12-\nu_2/4-\nu_3/3-\nu_{\infty}/2,
	\end{equation}
	where $\nu_2$ is the number of ramification points of order $2$, $\nu_3$ is the number of ramification points of order $3$, $\nu_{\infty}$ is the number of cusps and $m=[\Gamma(1):\Gamma]$.
\end{prop}
\begin{proof}
	See \cite[p. 23, Proposition 1.40]{ShimuraIntroToArith} or \cite[p. 37-38, Theorem 2.22]{JSMilne}.
\end{proof}

\subsection{Principle Congruence Subgroups}
Recall that our subject is the finitely punctured Riemann sphere, if we consider it as a quotient space $\mathbb{H}/\Gamma(N)$ for some suitable $N\in\mathbb{Z}$, the fact that the punctured Riemann sphere has genus $\it{zero}$ implies $g(\Gamma(N))=0<1$. The following theorem will show that $N=2,3,4,5$ are the only values such that $g(\Gamma(N))=0$.
\begin{theorem}\label{thm:genus-formula-Gamma-N}
	The genus of $\mathbb{H}/\Gamma(N)$ is given by the following formula:
	\begin{equation}\label{eq:genus-formula-Gamma-N}
	g(\Gamma(N))=\left\{\begin{array}{lc}
		0, & N=2,\\
		1+\frac{N-6}{24}\cdot N^2\prod_{p|N}(1-p^{-2}), & N\ge 3.
		\end{array}\right.
	\end{equation}
\end{theorem}

Recall that if $N>1$, $\Gamma(N)$ does not have \textit{elliptic} element, which is equivalent to the fact that $\mathbb{H}/\Gamma(N)$ does not contain any elliptic point, so it does not have any ramification point as well. Therefore $\nu_2=\nu_3=0$ in equation \eqref{eq:genus-formula-subgroup-of-Gamma1}, the genus is given by the following simplified formula 
\begin{equation}
g(\Gamma(N))=1+m/12-\nu_{\infty}/2,\qquad n\ge 2.\label{eq:genus-Gamma-N-formula-1}
\end{equation}

By the properties of elements in $\Gamma(N)$, the index $m$ of $\Gamma(N)$ in $\Gamma(1)$ is given by the following proposition. 
\begin{prop}\label{prop:principle-congruence-index}
	The index $m$ of a principle congruence group $\Gamma(N)$ in the full modular group is given by the formula
	\begin{equation}\label{eq:principle-congruence-index-1}
	m=(\Gamma(1):\Gamma(N))=\left\{\begin{array}{lc}
	6, & N=2,\\
	\frac{1}{2}\#(\mbox{SL}_2(\mathbb{Z}/N\mathbb{Z}))=\frac{1}{2}N^3\prod_{p|N}(1-p^{-2}), & N\ge 3,
	\end{array}\right.
	\end{equation}
	where $p$ are the prime divisors of $N$.
\end{prop}
\begin{proof}
	See \cite[p. 22, Equation (1.62)]{ShimuraIntroToArith} (where $\tilde{\Gamma}(N)$ is in our notation $\Gamma(N)$).
\end{proof}

Consider the stabilizer of $\infty$ in $\Gamma(1)$ and the stabilizer of $\infty$ in $\Gamma(N)$, we have the following proposition.
\begin{prop}\label{prop:number-cusps-principal-subgroup}
	The number of cusps of $\Gamma(N)$ is $\nu_{\infty}=m/N$, where $N\ge 2$.
\end{prop}
\begin{proof}
	See \cite[p. 22-23]{ShimuraIntroToArith}.
\end{proof}

Therefore equation \eqref{eq:genus-Gamma-N-formula-1} becomes
\begin{equation}
g(\Gamma(N))=1+\frac{m}{12N}(N-6),\label{eq:genus-Gamma-N-formula-2}
\end{equation}
equation \eqref{eq:genus-Gamma-N-formula-2} and \eqref{eq:principle-congruence-index-1} imply equation \eqref{eq:genus-formula-Gamma-N}. Therefore the number of cusps of $\Gamma(N)$ is given by the formula:
$$\nu_{\infty}(\Gamma(N))=m/N=\left\{\begin{array}{lc}
3, & N=2,\\
\frac{1}{2}N^2\Pi_{p|N}(1-p^{-2}), & N\ge 3.
\end{array}\right.$$
We list the number of cusps for each $\Gamma(N)$, $N=2,3,4,5$,
\begin{align}
&\nu_{\infty}(\Gamma(2))=\frac{m(\Gamma(2))}{2}=\frac{6}{2}=3, &\nu_{\infty}(\Gamma(3))&=\frac{1}{2}\cdot 9\cdot\frac{8}{9}=4,\notag\\
&\nu_{\infty}(\Gamma(4))=\frac{1}{2}\cdot 16\cdot\frac{3}{4}=6, &\nu_{\infty}(\Gamma(5))&=\frac{1}{2}\cdot 25\cdot\frac{24}{25}=12.\notag
\end{align}
From now on, we will use $n(N)$ to denote the number of cusps of $\Gamma(N)$. More precisely,
\begin{equation}
n(2)=3,\quad n(3)=4,\quad n(4)=6,\quad n(5)=12.
\end{equation}

Let us conclude the following theorem.

\begin{theorem}\label{thm:Gamma-Aut-sec-6}
	Let us define $n(N)$, $N=2,3,4,5,$ as above. If we have a covering space
	$$f:\mathbb{H}\rightarrow\mathbb{CP}^1\backslash\{a_1,a_2,\ldots,a_{_n(N)}\}=\mathbb{H}/\Gamma(N)$$
	for approporiate choices of $\{a_1,a_2,\ldots,a_{_{n(N)}}\}\subset\mathbb{CP}^1$, then $\Aut(f)=\Gamma(N)$.
\end{theorem}
\begin{proof}
	The definition of $\Gamma(N)$ implies that $\Gamma(N)$ does not contain elliptic elements. From Propositions \ref{prop:discreteneighborhood} and \ref{proposition-5.7}, for an arbitrary point $x\in\mathbb{H}$, there exists a neighborhood $x\in U$ such that $\gamma_1(U)\cap\gamma_2(U)=\emptyset$ for any $\gamma_1\neq\gamma_2\in\Gamma(N)$. The conclusion directly follows from \cite[p. 72, Proposition 1.40 (b)]{AllenHatcher}.
\end{proof}

    \section{Modular Forms and Modular Functions}\label{section:modular-forms}
\begin{definition}
	A function $f:\mathbb{H}\rightarrow\mathbb{C}$ is called a modular function of weight $k$ if $f$ satisfies the following properties:
	\begin{enumerate}
		\item The function $f$ is meromorphic on $\mathbb{H}$;
		\item The equation $f\circ\gamma(\tau)=(c\tau+d)^{k}f(\tau)$ holds for any $\tau\in\mathbb{H}$ and any $\gamma=\left(\begin{array}{cc}
		a & b\\
		c & d
		\end{array}\right)\in\Gamma(1)$, where $\Gamma(1)$ is the modular group;\\
		\item The function $f$ is meromorphic at infinity.
	\end{enumerate}
	Furthermore, if $f$ is a modular function of weight $k$ and holomorphic on $\mathbb{H}\cup\{\infty\}$, we call $f$ a modular form of weight $k$.
\end{definition}

As a matter of fact, the space of weight $4$ modular form for $\Gamma(1)$ is one dimensional and generated by the Eisenstein series $E_4(\tau)$. This is a classic result, see \cite[p. 88, Theorem 4 (ii) and p. 93, Examples]{SerreJP} (where $E_2$ is in our notation $E_4$).

\begin{theorem}\label{thm:weight-4-modular-form-E-4}
	The weight $4$ modular form is a dimension one vector space generated by the Eisenstein series $E_4(\tau)$. $E_4(\tau)$ has the following expansion
	\begin{equation}\label{eq:weight-4-modular-form-e-4}
    E_4(\tau)=1+240\sum_{m=1}^{\infty}\sigma_3(m)q^m
	\end{equation}
	in $q=\exp\{2\pi i\tau\}$, $\tau\in\mathbb{H}$, where $\sigma_3(m)=\sum_{d|m}d^3$-the sum of the cubes of all positive divisors of $m$.
\end{theorem}

\subsection{Schwarzian Derivative and Modular Function and Modular Form}
For convenience, we introduce the following definition.
\begin{definition}
	We say that a function $f:\mathbb{H}\rightarrow\mathbb{C}$ is an automorphic function for a discrete group $\Gamma'\subseteq\mbox{SL}_2(\mathbb{R})$ if $f$ is meromorphic on $\mathbb{H}\cup\{\infty\}$ and satisfies property
	\begin{equation*}
	f\circ\gamma(\tau)=f(\tau),\qquad\mbox{for all } \tau\in\mathbb{H} \mbox{ and }\gamma\in\Gamma'.
	\end{equation*}
\end{definition}

\begin{remark}
	Modular functions and modular forms are automorphic functions and automorphic forms of weight $zero$ for the modular group $\Gamma(1)$, respectively. Notice that the covering map $f:\mathbb{H}\rightarrow\punctured$ is an automorphic function, it is a generator of the function field over $\mathbb{H}/\Aut(f)$ which has a traditional name \textit{Hauptmodul}. We shall refer the covering map $f:\mathbb{H}\rightarrow \punctured$ as a Hauptmodul for the group $\Aut(f)$ in the future.
\end{remark}

The following lemma is mentioned in \cite[Proposition 3.2]{McKay2000}, for the completion of this article, I will restate and proof it.

\begin{lemma}\label{lemma:auto-func-dimension-one-fraction}
		Let $f$ be a Hauptmodul for a genus zero discrete group $\Aut(f)\subset \mbox{SL}_2(\mathbb{R})$. For every $\gamma$ that normalizes $\Aut(f)$ in $\mbox{SL}_2(\mathbb{R})$, there exists a corresponding matrix $\eta=\left(\begin{array}{cc}
	\eta_1 & \eta_2\\
	\eta_3 & \eta_4
	\end{array}\right)\in\mbox{SL}_2(\mathbb{C})$ such that the following equation
	\begin{equation*}
	f\circ\gamma(\tau)=\eta\circ f(\tau)
	\end{equation*}	
	is true for any $\tau\in\mathbb{H}$.
\end{lemma}
\begin{proof}
	Notice that $\gamma$ normalizes $Aut(f)$ in $\mbox{SL}_2(\mathbb{R})$, i.e., $\gamma\in\mathbb{H}=\mbox{SL}_2(\mathbb{R})$. Therefore the composition $f\circ\gamma(\tau)$ is also an automorphic function for $\Aut(f)$. The compactification of $\punctured\cong\mathbb{H}/\Aut(f)$ is the Riemann sphere $\mathbb{CP}^1$ which has genus zero, which implies that $\mathbb{H}^*/\Aut(f)$ has transcendental degree zero. Thus $f\circ\gamma(\tau)$ and $f(\tau)$ are related by an automorphism of $\mathbb{CP}^1$, i.e., an element $\eta$ in $\mbox{SL}_2(\mathbb{C})$. More precisely, assume $\gamma=\left(\begin{array}{cc}
	a & b\\
	c & d
	\end{array}\right)\in\mbox{SL}_2(\mathbb{R})$ and $\eta=\left(\begin{array}{cc}
	\eta_1 & \eta_2\\
	\eta_3 & \eta_4
	\end{array}\right)\in\mbox{SL}_2(\mathbb{C})$, we have the conclusion
	\begin{equation*}
	f\left(\frac{a\tau+b}{c\tau+d}\right)=\frac{\eta_1 f(\tau)+\eta_2}{\eta_3f(\tau)+\eta_4}
	\end{equation*}
	holds for every $\tau\in\mathbb{H}\cup\{\infty\}$.
\end{proof}

The following theorem plays an important role in determining the automorphic function for $\Gamma(N)$, $N=2,3,4,5$, it is also mentioned in \cite[Proposition 3.1]{McKay2000}, I restate and proof it here for the completion of this article. We will apply Proposition \ref{prop:schwarzian-basic-properties} and Lemma \ref{lemma:auto-func-dimension-one-fraction} to prove the following theorem.

\begin{theorem}\label{thm:schwarzian-f-tau-weight-4}
		Let $f$ be a Hauptmodul for a genus zero discrete group $\Gamma'\subseteq\mbox{SL}_2(\mathbb{R})$, then $\{f,\tau\}$ is a weight $4$ automorphic form for $\Gamma'$.
\end{theorem}
\begin{proof}
	First we show that $\{f,\tau\}$ is holomorphic on $\mathbb{H}$ and also holomorphic at infinity. From the assumption that $f(\tau)$ is a covering map, $f(\tau)$ is locally biholomorphic at any point $\tau_0\in\mathbb{H}$ with $f'(\tau_0)\neq0$. The definition of Schwarzian derivate
	$$\{f,\tau\}=2\left(\frac{f_{\tau\tau}}{f_{\tau}}\right)_{\tau}-\left(\frac{f_{\tau\tau}}{f_{\tau}}\right)^2$$
	shows the analyticity of $\{f,\tau\}$ at $\tau_0$. Recall equations \eqref{eq:schwarzian-expression-qk-tau-relation} and \eqref{eq:schwarzian-f-qk-expansion-2}, the following equation
	\begin{align}
	\{f,\tau\}&=\frac{4\pi^2}{k^2}(1-\qk^2\{f,\qk\}),\notag\\
	&=\frac{4\pi^2}{k^2}(1-\sum_{m=0}^{\infty}P_m(\B,\C_3,\ldots,\C_{m+3})\qk^{m+2}),\notag
	\end{align}
	implies that $\{f,\tau\}$ is analytic at $\tau=\infty$. Therefore, $\{f,\tau\}$ is holomorphic on $\mathbb{H}\cup\{\infty\}$. To show $\{f,\tau\}$ is a weight $4$ modular form for $\Gamma'$ we only need  to show equation	
	\begin{equation}\label{eq:schwarzian-weight-4-form-1}
	\{f(\gamma(\tau)),\gamma(\tau)\}=(c\tau+d)^4\{f(\tau),\tau\}
	\end{equation}
	holds for every element $\gamma=\left(\begin{array}{cc}
	a & b\\
	c & d
	\end{array}\right)\in\Gamma'$. On one hand, we know that $f\circ\gamma(\tau)$ and $f(\tau)$ are related by a linear transformation from Lemma \ref{lemma:auto-func-dimension-one-fraction}, and Proposition \ref{prop:schwarzian-basic-properties} implies equation
	\begin{align}
	\{f(\tau),\tau\}&=\{f(\tau),\gamma(\tau)\}(\gamma_{\tau})^2+\{\gamma(\tau),\tau\}=\{f(\tau),\gamma(\tau)\}(\gamma_{\tau})^2+\{\tau,\tau\},\notag\\
	&=\{f(\tau),\gamma(\tau)\}\cdot(c\tau+d)^4,\label{eq:schwarzian-weight-4-form-3}
	\end{align}
	where $\{\tau,\tau\}=0$ since $\tau_{\tau\tau}=(1)_{\tau}=0$. Therefore equation \eqref{eq:schwarzian-weight-4-form-1} holds, i.e., $\{f,\tau\}$ is a weight $4$ modular form for $\Gamma'$.
\end{proof}

Next we have the following conclusion as a corollary of Theorem \ref{thm:schwarzian-f-tau-weight-4}. 

\begin{coro}\label{coro:schwarzian-f-tau-weight-4-normalizer}
	ILet $f$ be a Hauptmodul for a genus zero discrete group $\Gamma'\subseteq\mbox{SL}_2(\mathbb{R})$. Then $\{f,\tau\}$ is a weight $4$ automorphic form for the normalizer of $\Gamma'$ in $\mbox{SL}_2(\mathbb{R})$.
\end{coro}
\begin{proof}
	Assume $\mu\in\mbox{SL}_2(\mathbb{R})$ normalizes $\Gamma'$. By Lemma \ref{lemma:auto-func-dimension-one-fraction}, there exists an matrix $\eta=\left(\begin{array}{cc}
	\eta_1 & \eta_2\\
	\eta_3 & \eta_4
	\end{array}\right)\in\mbox{SL}_2(\mathbb{C})$ such that $f(\tau)=\eta\circ f(\tau)$.
	Therefore we have the following equalities
	\begin{align}
	\{f(\mu(\tau)),\mu(\tau)\}&=\{\frac{\eta_1f(\tau)+\eta_2}{\eta_3f(\tau)+\eta_4},\mu(\tau)\}=\{f(\tau),\mu(\tau)\},\notag\\
	&=(c\tau+d)^4\{f,\tau\},\notag
	\end{align}
	where $\mu=\left(\begin{array}{cc}
	a & b\\
	c & d
	\end{array}\right)\in\Aut(f)$, and the last equation comes from direct calculation. This shows that $\{f,\tau\}$ is a weight $4$ auromorphic form for the normalizer of $\Gamma'$ in $\mbox{SL}_2(\mathbb{R})$.
\end{proof}

\begin{lemma}\label{lemma:gamma-N-gamma-1-normolizer}
	$\Gamma(1)$ normalizes $\Gamma(N)$ in $\mbox{SL}_2(\mathbb{R})$. Consequently, $\Gamma(N)$ is normal in $\Gamma(1)$.
\end{lemma}
\begin{proof}
	The proof is elementary.
\end{proof}
    \section{Main Result and Examples}

\subsection{Main Result}
In section \ref{section:qk-expansion-metric-formula}, we mentioned the inversion series \eqref{eq:qk-f-inversion}. Now let us consider the situation that one of the singularities is $a_1=0$, and the corresponding parabolic generator fixes infinity. It is equivalent to say that a covering map $f:\mathbb{H}\rightarrow\punctured$ has expansion 
\begin{equation*}
	f=A\qk+B\qk^2+c_3\qk^3+\sum_{m=4}^{\infty}c_m\qk^m
\end{equation*}
in $\qk=\exp\{\frac{2\pi i}{k}\tau\}$, $\tau\in\mathbb{H}$, for some real constant $k$ with $A\neq 0$. Let us denote $\f=\frac{f}{A}$, $f\in\punctured$, for convenience, recall equation \eqref{eq:coefficient-depend-on-A-B} in Theorem \ref{thm:coefficient-depend-on-A-B},
\begin{equation*}
	\f=f/A=\qk+\B\qk^2+\C_3(A,\B)\qk^3+\sum_{m=4}^{\infty}\C_m(A,\B)\qk^m,
\end{equation*}
where $\B=\frac{B}{A}$, $\C_m=\frac{c_m}{A}$ for $m\ge 3$. It is not hard to see that $\qk$ has the following expansion
\begin{equation}\label{eq:qk-f-inversion-poly-expression}
	\qk(\f)=\f+\tB(\B)\f^2+\tc_3(\B,\C_3)\f^3+\sum_{m=4}^{\infty}\tc_m(\B,\C_3,\ldots,\C_m)\f^m
\end{equation}
in $\f$, where $\tB(\B)=-\B$ and $\tc_m(\B,\C_3,\ldots,\C_m)$ are polynomials in $\B,\C_3,\ldots,\C_m$ which has degree $1$ in $\C_m$ with constant coefficient. Let us restate the main result, Theorem \ref{thm:into-main-1}, and proof it here.

\begin{theorem}\label{thm:final-metric-formula-expansion-1}
	Let $f:\mathbb{H}\rightarrow\mathbb{CP}^1\setminus\{a_1=0,a_2,\ldots,a_n\}$ be a covering map with expression
	\begin{equation}
	f=f(\tau)=A\qk+B\qk^2+c_3\qk^3+\sum_{m=4}^{\infty}c_m\qk^m.\notag
	\end{equation}
	Then the complete K\"{a}hler-Einstein metric has asymptotic expansion
	\begin{align}
	|ds|=\frac{1}{|A||\f|\log|\f|}\left|1-\left(\B\f-\frac{\Re(\B\f)}{\log|\f|}\right)+\sum_{m=2}^{\infty}R_m(A,\B,\f,\frac{\f^s\overline{\f^{m-s}}}{\log^j|\f|})\right||df|\label{eq:ds-final-expansion-in-auto-coefficients}
	\end{align}
	at the cusp $0$, where $\f=\frac{f}{A}$ for $f\in\mathbb{CP}^1\setminus\{a_1=0,a_2,\ldots,a_n\}$, $\B=\frac{B}{A}$, and $R_m(A,\B,\f,\frac{\f^s\overline{\f^{m-s}}}{\log^j|\f|})$ is a polynomial in $A,\B,\f,\frac{\f^s\overline{\f^{m-s}}}{\log^j|\f|}$, $s,j=0,1,\ldots,m$, with constant coefficients for $m\ge 2$.
\end{theorem}
\begin{proof}
	For convenience, we will simply write $q=\qk$ in this proof. Recall the proof of Theorem \ref{thm:metric-expression}, we apply the substitution $\f=\frac{f}{A}$,
	\begin{align}
	ds^2&=\frac{-4}{\left(\frac{k}{2\pi i}\log q(\f)-\overline{\frac{k}{2\pi i}\log q(\f)}\right)^2}\left|d\left(\frac{k}{2\pi i}\right)\log q(\f)\right|^2,\notag\\
	&=\frac{|q_{\f}(\f)|^2}{|q(\f)|^2\log^2|q(\f)|}\left|\frac{d\f}{df}\right||df|^2,\notag\\
	&=\frac{|q_{\f}(\f)|^2}{A^2|q(\f)|^2\log^2|q(\f)|}|df|^2.\label{eq:metric-expasion-in-bf(f)}
	\end{align}
	We calculate the expansion of each terms in equation \eqref{eq:metric-expasion-in-bf(f)} by \eqref{eq:qk-f-inversion-poly-expression},
	\begin{align}
	\frac{q_{_{\f}}(\f)}{q(\f)}&=(\log|q(\f)|)_{\f}=(\log|\f|+\log|1+\tB\f+\sum_{m=2}^{\infty}\tc_{m+1}\f^m|)_{_{\f}},\notag\\
	                        &=\frac{1}{\f}+\tB+\sum_{m=1}^{\infty}\tilde{Q}^{(11)}_m(\tB,\tc_3,\ldots,\tc_{m+2})\f^m,\label{eq:q'-over-q-expansion}
	\end{align}
    where $\tilde{Q}^{(11)}_m(\tB,\tc_3,\ldots,\tc_{m+2})$ is a polynomial in $\tB,\tc_3,\ldots,\tc_{m+2}$ which has degree $1$ in $\tc_{m+2}$ with constant coefficients for $m\ge 1$. Next we calculate the expansion of $\frac{1}{\log|q(\f)|}$ in $\f$ and $\log|\f|$,
	\begin{align}
	\frac{1}{\log|q(\f)|}
	&=\frac{1}{\log|\f|}\left(1+\frac{\log|1+\tB\f+\sum_{m=2}^{\infty}\tc_{m+1}\f^m|}{\log|\f|}\right)^{-1},\notag\\
	&=\frac{1}{\log|\f|}\left[1-\frac{\log|1+\tB\f+\sum_{m=2}^{\infty}\tc_{m+1}\f^m|}{\log|\f|}+\sum_{l=2}^{\infty}(-1)^l\left(\frac{\log|1+\tB\f+\sum_{m=2}^{\infty}\tc_{m+1}\f^m|}{\log|\f|}\right)^l\right].\label{eq:1/log|q|-expansion-part-1}
	\end{align}
	Let us use the expansion $\log(1+x)=x-\frac{x^2}{2}+\frac{x^3}{3}+\sum_{n\ge 4}(-1)^{n+1}\frac{x^n}{n}$ and write $\tc_1=1$ and $\tc_2=\tB$ for convenience,
	\begin{align}
	\log|1+\tB\f+\sum_{m=2}^{\infty}\tc_{m+1}\f^m|
	&=\frac{1}{2}\log|1+\tB\f+\tc_3\f^2+\sum_{m=2}^{\infty}\tc_{m+1}\f^m|^2,\notag\\
	&=\frac{1}{2}\log\left[1+2\Re(\tB\f)+\sum_{m=2}^{\infty}\left(\sum_{s+j=m,s,j\ge 0}\tc_{s+1}\overline{\tc_{j+1}}\f^s\overline{\f^j}\right)\right],\notag\\
	&=\frac{1}{2}\left\{\left[2\Re(\tB\f)+\sum_{m=2}^{\infty}\left(\sum_{s+j=m,s,j\ge 0}\tc_{s+1}\overline{\tc_{j+1}}\f^s\overline{\f^j}\right)\right]\right.\notag\\
	&\qquad\qquad\left.+\sum_{l=2}^{\infty}(-1)^{l+1}\frac{1}{l}\left[2\Re(\tB\f)+\sum_{m=2}^{\infty}\left(\sum_{s+j=m,s,j\ge 0}\tc_{s+1}\overline{\tc_{j+1}}\f^s\overline{\f^j}\right)\right]^l\right\},\notag\\
	&=\frac{1}{2}\left[2\Re(\tB\f)+\sum_{m=2}^{\infty}R^{(1)}_m(\f^m,\f^{m-1}\overline{\f},\ldots,\f\overline{\f^{m-1}},\overline{\f^m})\right],\notag
	\end{align}
	where $R^{(1)}_m(\f^m,\f^{m-1}\overline{\f},\ldots,\f\overline{\f^{m-1}},\overline{\f^m})$ is a polynomial in $\f^m,\f^{m-1}\overline{\f},\ldots,\f\overline{\f^{m-1}},\overline{\f^m}$ with coefficients in $\tB,\tc_3,\ldots,\tc_{m+1}$. Therefore equation \eqref{eq:1/log|q|-expansion-part-1} has expansion
	\begin{align}
	\frac{1}{\log|q(\f)|}
	&=\frac{1}{\log|\f|}\left(1-\frac{1}{2}\frac{1}{\log|\f|}\left[2\Re(\tB\f)+\sum_{m=2}^{\infty}R^{(1)}_m(\f^m,\f^{m-1}\overline{\f},\ldots,\f\overline{\f^{m-1}},\overline{\f^m})\right]\right.\notag\\
	&\qquad\qquad \left.+\sum_{l=2}^{\infty}\frac{(-1)^l}{2^l\log^l|\f|}\left[2\Re(\tB\f)+\sum_{m=2}^{\infty}R^{(1)}_m(\f^m,\f^{m-1}\overline{\f},\ldots,\f\overline{\f^{m-1}},\overline{\f^m})\right]^l\right),\notag\\
	&=\frac{1}{\log|\f|}\left[1-\frac{\Re(\tB\f)}{\log|\f|}+\sum_{m=2}^{\infty}R^{(2)}_m(\frac{\f^s\overline{\f^{m-s}}}{\log^j|\f|})\right],\label{eq:1/log|q|-expansion-part-final}
	\end{align}
	where $R^{(2)}_m(\frac{\f^s\overline{\f^{m-s}}}{\log^j|\f|})$ is a polynomial in $\frac{\f^s\overline{\f^{m-s}}}{\log^j|\f|}$, $s=0,1,\ldots,m$ and $j=1,\ldots,m$, which has coefficients in $\tB,\tc_3,\ldots,\tc_{m+1}$ for $m\ge 2$. Therefore equation \eqref{eq:q'-over-q-expansion} and \eqref{eq:1/log|q|-expansion-part-final} implies that the metric $|ds|$ defined by equation \eqref{eq:metric-expasion-in-bf(f)} has expression
	\begin{align}
	|ds|&=\frac{1}{|A||\f|\log|\f|}\left|1+\tB\f+\sum_{m=1}^{\infty}\tilde{Q}^{(11)}(\tB,\tc_3,\ldots,\tc_{m+2})\f^{m+1}\right|\cdot\left[1-\frac{\Re(\tB\f)}{\log|\f|}+\sum_{m=2}^{\infty}R^{(2)}_m(\frac{\f^s\overline{\f^{m-s}}}{\log^j|\f|})\right],\notag\\
	&=\frac{1}{|A||\f|\log|\f|}\left|1+\left(\tB\f-\frac{\Re(\tB\f)}{\log|\f|}\right)+\sum_{m=2}^{\infty}\tilde{R}_m(\f,\frac{\f^s\overline{\f^{m-s}}}{\log^j|\f|})\right||df|,
	\end{align}
	where $\tilde{R}_m(\f,\frac{\f^s\overline{\f^{m-s}}}{\log^j|\f|})$ is a polynomial in the variable set $\frac{\f^s\overline{\f^{m-s}}}{\log^j|\f|}$, $s,j=0,1,\ldots,m$, with coefficients in $\tB,\tc_2,\ldots,\tc_{m+1}$ for $m\ge 2$. Due to equation \eqref{eq:qk-f-inversion-poly-expression} and Theorem \ref{thm:coefficient-depend-on-A-B}, each term $\tB,\tc_2,\ldots,\tc_{m+1}$ can be solved as a polynomial in $A,\B$, so the coefficients in $\tilde{R}_m(\f,\frac{\f^s\overline{\f^{m-s}}}{\log^j|\f|})$ are polynomials in $A,\B$. Let us write $\tilde{R}_m(\f,\frac{\f^s\overline{\f^{m-s}}}{\log^j|\f|})$ as $R_m(A,\B,\f,\frac{\f^s\overline{\f^{m-s}}}{\log^j|\f|})$, which denotes a polynomial in $A,\B,\f,\frac{\f^s\overline{\f^{m-s}}}{\log^j|\f|}$ with constant coefficients. Recall that $\tB=-\B$, we have conclusion that the metric $|ds|$ is given by the following expression
	\begin{equation}
       |ds|=\frac{1}{|A||\f|\log|\f|}\left|1-\left(\B\f-\frac{\Re(\B\f)}{\log|\f|}\right)+\sum_{m=2}^{\infty}R_m(A,\B,\f,\frac{\f^s\overline{\f^{m-s}}}{\log^j|\f|})\right||df|,
	\end{equation}
	and, consequently, $|ds|$ is uniquely determined up to a choice of $A,\B$.
\end{proof}

Recall Theorem \ref{thm:Gamma-Aut-sec-6}, now let us focus on the case that the deck transformation group is $\Gamma(N)$, $N=2,3,4,5$.
\begin{theorem}\label{thm:2nd-main-thm}
	Let $n(N)$ be the numbers that is defined by equation \eqref{eq:cusp-number-Gamma-N}, and let $f_N:\mathbb{H}\rightarrow\Npunctured$ be a universal covering with deck transformation group $\Aut(f_N)=\Gamma(N)$, $N=2,3,4,5$, satisfying that $f_N$ vanishes at infinity, i.e., $f_N(\infty)=0$. Then $f_N$ can be given by the following expansion 
\begin{equation}\label{eq:2nd-main-theorem-equation}
	f_N(\tau)=A\q+B\q^2+\sum_{m=3}^{\infty}A\cdot\C_{m}(\B)\q^m
	\end{equation}
	in $\q=\exp\{\frac{2\pi}{N}i\tau\},\tau\in\mathbb{H}$, where the constants $A,B\in\mathbb{C}$ are uniquely determined by the set of values of the punctured points $\{a_1=0,a_2,\ldots,a_{n(N)}\}$, and the term $\C_m(\B)$ in the coefficient is a polynomial in $\B$ with constant coefficients for $m\ge 3$.
\end{theorem}
\begin{proof}
	By Proposition \ref{prop:schwarzian-expression-qk-tau-relation}, the following identity
	\begin{equation}
		\{f_N,\tau\}=\frac{4\pi^2}{N^2}(1-\q^2\{f_N,\q\})\notag
	\end{equation}
	holds. Corollary \ref{coro:schwarzian-f-tau-weight-4-normalizer} and Lemma \ref{lemma:gamma-N-gamma-1-normolizer} imply that $\{f_N,\tau\}$ is a weight 4 modular form for $\Gamma(1)$. Recall Theorem \ref{thm:weight-4-modular-form-E-4}, there is a constant $\kappa$ such that the following equation
	\begin{equation}
	E_4(\tau)=\kappa\{f_N,\tau\}=\kappa\frac{4\pi^2}{N^2}(1-\q^2\{f_N,\q\})
	\end{equation}
	holds. Equations \eqref{eq:weight-4-modular-form-e-4} and \eqref{eq:schwarzian-f-qk-expansion-2} imply the following identity
	\begin{equation}\label{eq:main-theorem-schwarzian-qN-expansion}
	1+240\sum_{m=1}^{\infty}\sigma_3(m)q^m=\frac{4\pi^2}{N^2}\kappa[1-\sum_{m=0}^{\infty}P_m(\B,\C_3,\ldots,\C_{m+3})\q^{m+2}].
	\end{equation}
	Matching the constant terms in equation \eqref{eq:main-theorem-schwarzian-qN-expansion},
	\begin{equation*}
		1=\frac{4\pi^2}{N^2}\kappa\quad\mbox{implies}\quad\kappa=\frac{N^2}{4\pi^2}.
	\end{equation*}
	Therefore we have the following identity
	\begin{equation*}
		1+240\sum_{m=1}^{\infty}\sigma_3(m)q^m=1-P_0(\B,\C_3)\q^2-\sum_{m=1}^{\infty}P_m(\B,\C_3,\ldots,\C_{m+3})\q^{m+2}
	\end{equation*}
	holds. Notice the identity $\q^N=q$ by their definitions, we match the coefficient of $\q^{lN}$ in equation \eqref{eq:main-theorem-schwarzian-qN-expansion} with the coefficient of $q^l$ in equation \eqref{eq:weight-4-modular-form-e-4}, and let other coefficients in equation \eqref{eq:main-theorem-schwarzian-qN-expansion} be $0$. We get the following set of equations,
	\begin{equation}
		P_m(\B,\C_3,\ldots,\C_{m+3})=\left\{
			\begin{array}{ll}
			240\sigma_3(l), & \mbox{if }m=lN-2\mbox{ for }l=1,2,\ldots,\\
			0,              & \mbox{otherwise},
			\end{array}\right.
	\end{equation}
	where $m\ge 0$. Therefore $\C_3$ can be solved in terms of $\B$ when $m=0$, and notice that every coefficient $P_m(\B,\C_3,\ldots,\C_{m+3})$ is a polynomial of degree $1$ in $\C_{m+3}$ with constant coefficient. By induction, we can conclude that $\C_m$ can be solved as a polynomial in $\B$. Equation \eqref{eq:2nd-main-theorem-equation} holds since $c_m=A\cdot\C_m$, $m\ge 3$.
\end{proof}

\begin{coro}\label{coro:gamma-N-metric-expansion}
Under the same assumption in Theorem \ref{thm:2nd-main-thm}, the complete K\"{a}hler-Einstein metric has asymptotic expansion
\begin{equation}
|ds|=\frac{1}{|A||\f|\log|\f|}\left|1-\left(\B\f-\frac{\Re(\B\f)}{\log|\f|}\right)+\sum_{m=2}^{\infty}R_m(\B,\f,\frac{\f^s\overline{\f^{m-s}}}{\log^j|\f|})\right||df|
\end{equation}
at the cusp $0$, where $R_m(\B,\f,\frac{\f^s\overline{\f^{m-s}}}{\log^j|\f|})$ is a polynomial in $\B,\f,\frac{\f^s\overline{\f^{m-s}}}{\log^j|\f|}$, $s,j=0,1,\ldots,m$, with constant coefficients for $m\ge 2$.
\end{coro}
\begin{proof}
It directly follows from Theorems \ref{thm:final-metric-formula-expansion-1} and \ref{thm:2nd-main-thm}.
\end{proof}
    \subsection{Examples}\label{section:example}
We will see some examples by applying the main theorems.
\begin{example}[The covering space $\mathbb{CP}^1\setminus\{a_1=0,a_2,\ldots,a_{12}\}\cong\mathbb{H}/\Gamma(5)$]
	In this case, $N=5$, $q_{_5}^5=q$, so the coefficients of $q_{_5}^2,q_{_5}^3,q_{_5}^4$ are all $0$. We have the following,
	$$
	\left\{\begin{array}{l}
	12(\C_3-\B^2)=0,\\
	48(\C_4-2\C_3\B+\B^3)=0,\\
	24(5\C_5-10\C_4\B+17\C_3\B^2-6\C_3^2-6\B^4)=0,\\
	\ldots,
	\end{array}\right.\quad\mbox{implies}\quad
	\left\{\begin{array}{c}
	\C_3=\B^2,\\
	\C_4=\B^3,\\
	\C_5=\B^4,\\
	\ldots.
	\end{array}\right.
	$$
	Therefore the covering map $f_5:\mathbb{H}\rightarrow\mathbb{CP}^1\setminus\{a_1=0,a_2,\ldots,a_{12}\}$ with deck transformation group $\Gamma(5)$ can be given by the following expansion
	\begin{equation}
	\f_5=\frac{f_5(\tau)}{A}=q_{_5}+\B q^2_{_5}+\B^2q^3_{_5}+\B^3q^4_{_5}+\B^4q^5_{_5}+\sum_{m=6}^{\infty}P_m(\B)q^m_{_5},
	\end{equation}
	where the value of $\B$ is up to the collection of punctures $\{a_1=0,a_2,\ldots,a_{12}\}$. Therefore the complete K\"{a}hler-Einstein metric at the cusp $a_1=0$ is given by the following equation
	\begin{align*}
	|ds|=\frac{1}{|A||\f|\log|\f|}\left|1-\left(\B\f-\frac{\Re(\B\f)}{\log|\f|}\right)+\sum_{m=2}^{\infty}R_m(\B,\f,\frac{\f^s\overline{\f^{m-s}}}{\log^j|\f|})\right||df|,
	\end{align*}
	where $\f=\f_5=\frac{f_5}{A}$ for convenience,  $f_5\in\mathbb{CP}^1\setminus\{a_1=0,a_2,\ldots,a_{12}\}$.
\end{example}

The following example is the case of the triple punctured Riemann sphere, which was mentioned in Example \ref{example:intro-thm-2}.
\begin{example}[The covering space of $\mathbb{H}/\Gamma(2)\cong\mathbb{CP}^1\setminus\{a_1=0,a_2,a_3\}$]\label{example:3-points-arbitrary}
	In this example, $N=2$, $q^2_{_2}=q$, we have the following,
	\begin{equation*}
	\left\{\begin{array}{l}
	12(\C_3-\B^2)=-240,\\
	48(\C_4-2\C_3\B+\B^3)=0,\\
	24(5\C_5-10\C_4\B+17\C_3\B^2-6\C_3^2-6\B^4)=-2160,\\
	\ldots,
	\end{array}\right.\quad\mbox{implies}\quad\left\{\begin{array}{l}
	\C_3=\B^2-20,\\
	\C_4=\B^3-40\B,\\
	\C_5=\B^4-60\B^2+462,\\
	\ldots.
	\end{array}\right.
	\end{equation*} 
	Consequently, the covering map $f_2:\mathbb{H}\rightarrow\mathbb{CP}^1\setminus\{a_1=0,a_2,a_3\}$ with deck transformation group $\Gamma(2)$ has expression
	\begin{equation}\label{eq:f-2-general-q-2-expansion}
	\f_2=\frac{f_2(\tau)}{A}=q_{_2}+\B q^2_{_2}+(\B^2-20)q^3_{_2}+\sum_{m=4}^{\infty}P_m(\B)q^m_{_2}.
	\end{equation}
	The metric equation \eqref{eq:ds-final-expansion-in-auto-coefficients} have the following expression
	\begin{align*}
	|ds|=\frac{1}{|A||\f|\log|\f|}\left|1-\left(\B\f-\frac{\Re(\B\f)}{\log|\f|}\right)+\sum_{m=2}^{\infty}R_m(\B,\f,\frac{\f^s\overline{\f^{m-s}}}{\log^j|\f|})\right||df|,
	\end{align*}
	where $\f=\f_2=\frac{f_2}{A}$, $f_2\in\mathbb{CP}^1\setminus\{0,a_2,a_3\}$, for convenience, and $\B=\frac{B}{A}$ is the only free parameter which is uniquely determined by the two punctures $a_2$ and $a_3$. 
\end{example}

Next example is a special case of the triple punctured Riemann sphere, which was mentioned as Example \ref{example:intro} in section \ref{sec:introduction}. The explicit metric formula is also given by S. Agard from a different approach in \cite{AgardDist}.

\begin{example}[The quotient space $\mathbb{H}/\Gamma(2)\cong\mathbb{CP}^1\setminus\{0,1,\infty\}$]\label{example:metric}
	When $A=16, B=-128$, equation \eqref{eq:f-2-general-q-2-expansion} is the covering map given by the modular lambda function
	\begin{equation}
	f_2=f_2(\tau)=\lambda(\tau)=16\qsec-128\qsec^2+704\qsec^3-3072\qsec^4+O(\qsec^5),
	\end{equation}
	where $f_2=\lambda(\tau)$ is the covering map of $\mathbb{H}\rightarrow\mathbb{CP}^1\setminus\{0,\infty,1\}$ with values at the cusps as below
	$$\infty\mapsto 0,\qquad0\mapsto 1,\qquad1\mapsto\infty.$$
	In this case, $\B=\frac{-128}{16}=-8$, $\C_3=\frac{704}{16}=44=(-8)^2-20$, the metric is given by the following
	\begin{align*}
	|ds|&=\frac{1}{|16||\f|\log|\f|}\left|1+8\left(\f-\frac{\Re\f}{\log|\f|}\right)\right.\\
	&\qquad\left.-\left[(2\cdot 44+5\cdot 64)\f^2-64\frac{\f\Re\f}{\log|\f|}+(44+\frac{5}{2}\cdot 64)\frac{\Re(\f^2)}{\log|\f|}+64\frac{(\Re\f)^2}{\log^2|\f|}\right]+O(\f^3)\right||df|,\\
	&=\frac{1}{|f|\log|f/16|}\left|1+\frac{1}{2}\left(f-\frac{\Re f}{\log|f/16|}\right)\right.\\
	&\qquad\left.-\left[\frac{51}{32}f^2-\frac{1}{4}\frac{f\Re f}{\log|f/16|}+\frac{51}{64}\frac{\Re(f^2)}{\log|f/16|}+\frac{1}{4}\frac{(\Re f)^2}{\log^2|f/16|}\right]+O(f^3)\right||df|,
	\end{align*}
	where $f=f_2$ and $\f=\f_2=\frac{f_2}{16}$, $f=f_2\in\mathbb{CP}^1\setminus\{0,1,\infty\}$.
\end{example}

\begin{example}[The punctured Riemann sphere $\mathbb{CP}^1\setminus\{a_1,a_2,a_3\}$ for arbitrary $a_1,a_2,a_3$]
	Assume $a_1, a_2$ and $a_3$ are three different points on $\mathbb{CP}^1$, the M\"{o}bius transformation 
	\begin{equation}
	\lambda\mapsto\frac{a_1(a_2-a_3)-a_3(a_2-a_1)\lambda}{(a_2-a_3)-(a_2-a_1)\lambda}=\tilde{\lambda}\notag 
	\end{equation}
	maps $\{0,1,\infty\}$ to $\{a_2,a_3,a_1\}$ respectively. Direct calculation indicates the following conclusion.
	\begin{coro}\label{coro:section-8-1}
		A covering map $f:\mathbb{H}\rightarrow\mathbb{CP}^1\setminus\{a_1,a_2,a_3\}$ for any three different points $a_1,a_2,a_3$ can be uniquely determined by the following equation
		\begin{equation}
		f(\tau)=\frac{a_1(a_2-a_3)-a_3(a_2-a_1)\lambda(\tau)}{(a_2-a_3)-(a_2-a_1)\lambda(\tau)}.
		\end{equation}
		Furthermore, $f(\tau)=\tilde{\lambda}(\tau)$ can be given by the following expansion,
		\begin{equation}
		f(\tau)=a_1+16\frac{(a_1-a_3)(a_2-a_1)}{a_2-a_3}\qsec+128(a_1-a_3)\left[2\frac{(a_2-a_1)^2}{(a_2-a_3)^2}-\frac{a_2-a_1}{a_2-a_3}\right]\qsec^2+O(\qsec^3).\notag
		\end{equation}
	\end{coro}	
	Therefore the metric expansion at $a_1$ can be given by the following corollary.
	\begin{coro}\label{coro:section-8-2}
		The complete K\"{a}hler-Einstein
		metric on $\mathbb{CP}^1\setminus\{a_1,a_2,a_3\}$ at cusp $a_1$ has the following asymptotic expansion
		\begin{equation}
		|ds|=\frac{|a_2-a_3|}{16|a_1-a_3||a_2-a_1|}\frac{1}{\mathfrak{f}\log|\mathfrak{f}|}\left\{1-8\left[(2\frac{a_2-a_1}{a_2-a_3}-1)\mathfrak{f}-\frac{\Re((2\frac{a_2-a_1}{a_2-a_3}-1)\mathfrak{f})}{\log|\mathfrak{f}|}\right]+O(\mathfrak{f}^2)\right\}|df|\notag
		\end{equation}
		in $\mathfrak{f}=\frac{(a_2-a_3)}{16(a_1-a_3)(a_2-a_1)}(f-a_1)$, $f\in\mathbb{CP}^1\setminus\{a_1,a_2,a_3\}$.
	\end{coro}
	\begin{proof}
		Direct calculation gives the value of coefficients $A, B$,
		$$A=16\frac{(a_1-a_3)(a_2-a_1)}{a_2-a_3},\,B=128(a_1-a_3)\left[2\frac{(a_2-a_1)^2}{(a_2-a_3)^2}-\frac{a_2-a_1}{a_2-a_3}\right].$$
		It implies the following value
		$$\B=\frac{B}{A}=8\left(2\frac{a_2-a_1}{a_2-a_3}-1\right).$$
		Then the result follows directly from Corollary \ref{coro:gamma-N-metric-expansion}.
	\end{proof}
\end{example}

  \section*{Acknowledgement}
   Special thanks to my advisor Professor Damin Wu, who has provided me incrediblly great instruction and support. The author would like to thank Professor Shing-Tung Yau, Professor Lizhen Ji and Professor Vladimir Markovic for helpful comments. Last but not least, the auother would like to thank the anonymous referees for their careful reviews and suggestions.


\begin{bibdiv}
	\begin{biblist}
		
		\bib{AgardDist}{article}{
		author={Agard, Stephen},
		title={Distortion theorems for quasiconformal mappings},
		date={1968},
		journal={Ann. Acad. Sci. Fenn. Ser. A I No. },
		number={413},
		pages={12 pages},
		review={\MR{0222288}},
	}
		
		\bib{AhlforsComplex}{book}{
			author={Ahlfors, Lars~V.},
			title={Complex analysis},
			edition={Third},
			publisher={McGraw-Hill Book Co., New York},
			date={1978},
			ISBN={0-07-000657-1},
			review={\MR{510197}},
		}
		
		\bib{AhlforsTeichmuller}{inproceedings}{
			author={Ahlfors, Lars~V.},
			title={Quasiconformal mappings, {T}eichm\"{u}ller spaces, and {K}leinian
				groups},
			date={1980},
			booktitle={Proceedings of the {I}nternational {C}ongress of {M}athematicians
				({H}elsinki, 1978)},
			publisher={Acad. Sci. Fennica, Helsinki},
			pages={71\ndash 84},
			review={\MR{562598}},
		}
		
		\bib{BersUniModuliKlein}{article}{
			author={Bers, L.},
			title={Uniformization. {M}oduli and {K}leinian groups},
			date={1973},
			ISSN={0042-1316},
			journal={Uspehi Mat. Nauk},
			volume={28},
			number={4(172)},
			pages={153\ndash 198},
			note={Translated from the English (Bull. London Math. Soc. {{\bf{4}}}
				(1972), 257--300) by A. Ju. Geronimus},
			review={\MR{0385085}},
		}
		
		\bib{BersTechmuller}{article}{
			author={Bers, Lipman},
			title={On boundaries of teichm\"{u}ller spaces and on {K}leinian groups.
				{I}},
			date={1970},
			ISSN={0003-486X},
			journal={Ann. of Math. (2)},
			volume={91},
			pages={570\ndash 600},
			url={https://doi.org/10.2307/1970638},
			review={\MR{0297992}},
		}
		
		\bib{BumpAuto}{book}{
			author={Bump, Daniel},
			title={Automorphic forms and representations},
			series={Cambridge Studies in Advanced Mathematics},
			publisher={Cambridge University Press, Cambridge},
			date={1997},
			volume={55},
			ISBN={0-521-55098-X},
			url={https://doi-org.ezproxy.lib.uconn.edu/10.1017/CBO9780511609572},
			review={\MR{1431508}},
		}
		
		\bib{ChoQianKobayashi}{article}{
		author={Cho, Gunhee},
		author={Qian, Junqing},
		title={The Kobayashi-Royden metric on punctured spheres},
		date={2020},
		journal={Forum Mathematicum},
		doi={https://doi.org/10.1515/forum-2019-0297},
	}
		
		\bib{KleinBook}{book}{
			author={F.~Klein}, 
			author={R.~Fricke},
			title={Lectured on the theory of automorphic functions},
			publisher={Higher Education Press(Beijing)},
			date={2017},
			ISBN={978-7-04-047840-2},
		}
		
		\bib{GriffithsZariski1971}{article}{
			author={Griffiths, Phillip~A.},
			title={Complex-analytic properties of certain {Z}ariski open sets on
				algebraic varieties},
			date={1971},
			ISSN={0003-486X},
			journal={Ann. of Math. (2)},
			volume={94},
			pages={21\ndash 51},
			url={https://doi-org.ezproxy.lib.uconn.edu/10.2307/1970733},
			review={\MR{0310284}},
		}
	
	    \bib{AllenHatcher}{book}{
	    author={Hatcher, Allen},
	    title={Algebraic topology},
	    publisher={Cambridge University Press, Cambridge},
	    date={2002},
	    pages={xii+544},
	    isbn={0-521-79160-X},
	    isbn={0-521-79540-0},
	    review={\MR{1867354}},
    }
		
		\bib{HubbardTeichmuller}{book}{
			author={Hubbard, John~Hamal},
			title={Teichm\"{u}ller theory and applications to geometry, topology,
				and dynamics. {V}ol. 1},
			publisher={Matrix Editions, Ithaca, NY},
			date={2006},
			ISBN={978-0-9715766-2-9; 0-9715766-2-9},
			review={\MR{2245223}},
		}
		
		\bib{IrwinKraAccessoryPara}{article}{
			author={Kra, Irwin},
			title={Accessory parameters for punctured spheres},
			date={1989},
			ISSN={0002-9947},
			journal={Trans. Amer. Math. Soc.},
			volume={313},
			number={2},
			pages={589\ndash 617},
			url={https://doi.org/10.2307/2001420},
			review={\MR{958896}},
		}
		
		\bib{McKay2000}{article}{
			author={McKay, John},
			author={Sebbar, Abdellah},
			title={Fuchsian groups, automorphic functions and Schwarzians},
			date={2000},
			journal={Mathematische Annalen},
			volume={318},
			pages={255\ndash 275},
		}
		
		\bib{JSMilne}{misc}{
			author={Milne, James~S.},
			title={Modular functions and modular forms (v1.31)},
			date={2017},
			note={Available at www.jmilne.org/math/},
		}
		
		\bib{NehariConformalMap}{book}{
			author={Nehari, Zeev},
			title={Conformal mapping},
			publisher={Dover Publications, Inc., New York},
			date={1975},
			note={Reprinting of the 1952 edition},
			review={\MR{0377031}},
		}
		
		\bib{NevaROLF}{book}{
			author={Nevanlinna, Rolf},
			title={Analytic functions},
			series={Translated from the second German edition by Phillip Emig. Die
				Grundlehren der mathematischen Wissenschaften, Band 162},
			publisher={Springer-Verlag, New York-Berlin},
			date={1970},
			review={\MR{0279280}},
		}
		
		\bib{RochonZhang2012}{article}{
			author={Rochon, Fr\'{e}d\'{e}ric},
			author={Zhang, Zhou},
			title={Asymptotics of complete {K}\"{a}hler metrics of finite volume on
				quasiprojective manifolds},
			date={2012},
			ISSN={0001-8708},
			journal={Adv. Math.},
			volume={231},
			number={5},
			pages={2892\ndash 2952},
			url={https://doi.org/10.1016/j.aim.2012.08.005},
			review={\MR{2970469}},
		}
		
		\bib{Schumacher1998MathAnn}{article}{
			author={Schumacher, Georg},
			title={Asymptotics of {K}\"{a}hler-{E}instein metrics on
				quasi-projective manifolds and an extension theorem on holomorphic maps},
			date={1998},
			ISSN={0025-5831},
			journal={Math. Ann.},
			volume={311},
			number={4},
			pages={631\ndash 645},
			url={https://doi.org/10.1007/s002080050203},
			review={\MR{1637968}},
		}
		
		\bib{SebbarTorsionFree}{article}{
			author={Sebbar, Abdellah},
			title={Torsion-free genus zero congruence subgroups of {${\rm
						PSL}_2(\Bbb R)$}},
			date={2001},
			ISSN={0012-7094},
			journal={Duke Math. J.},
			volume={110},
			number={2},
			pages={377\ndash 396},
			url={https://doi-org.ezproxy.lib.uconn.edu/10.1215/S0012-7094-01-11028-4},
			review={\MR{1865246}},
		}
		
		\bib{SebbarModularCurve}{article}{
			author={Sebbar, Abdellah},
			title={Modular subgroups, forms, curves and surfaces},
			date={2002},
			ISSN={0008-4395},
			journal={Canad. Math. Bull.},
			volume={45},
			number={2},
			pages={294\ndash 308},
			url={https://doi.org/10.4153/CMB-2002-033-1},
			review={\MR{1904094}},
		}
		
		\bib{SerreJP}{book}{
			author={Serre, J.-P.},
			title={A course in arithmetic},
			publisher={Springer-Verlag, New York-Heidelberg},
			date={1973},
			note={Translated from the French, Graduate Texts in Mathematics, No.
				7},
			review={\MR{0344216}},
		}
		
		\bib{ShimuraIntroToArith}{book}{
			author={Shimura, Goro},
			title={Introduction to the arithmetic theory of automorphic functions},
			publisher={Publications of the Mathematical Society of Japan, No. 11.
				Iwanami Shoten, Publishers, Tokyo; Princeton University Press, Princeton,
				N.J.},
			date={1971},
			note={Kan\^{o} Memorial Lectures, No. 1},
			review={\MR{0314766}},
		}
		
		\bib{SteinFourier}{book}{
			author={Stein, Elias~M.},
			author={Shakarchi, Rami},
			title={Fourier analysis},
			series={Princeton Lectures in Analysis},
			publisher={Princeton University Press, Princeton, NJ},
			date={2003},
			volume={1},
			ISBN={0-691-11384-X},
			review={\MR{1970295}},
		}
		
		\bib{VenkovRussianPaper}{article}{
			author={Venkov, A.~B.},
			title={Accessory coefficients of a second-order {F}uchsian equation with
				real singular points},
			date={1983},
			ISSN={0373-2703},
			journal={Zap. Nauchn. Sem. Leningrad. Otdel. Mat. Inst. Steklov. (LOMI)},
			volume={129},
			pages={17\ndash 29},
			review={\MR{703005}},
		}
		
		\bib{WangSmoothSiegel1993}{article}{
			author={Wang, Wenxiang},
			title={On the smooth compactification of {S}iegel spaces},
			date={1993},
			ISSN={0022-040X},
			journal={J. Differential Geom.},
			volume={38},
			number={2},
			pages={351\ndash 386},
			url={http://projecteuclid.org.ezproxy.lib.uconn.edu/euclid.jdg/1214454298},
			review={\MR{1237488}},
		}
		
		\bib{DaminWu2006}{article}{
			author={Wu, Damin},
			title={Higher canonical asymptotics of {K}\"{a}hler-{E}instein metrics
				on quasi-projective manifolds},
			date={2006},
			ISSN={1019-8385},
			journal={Comm. Anal. Geom.},
			volume={14},
			number={4},
			pages={795\ndash 845},
			url={http://projecteuclid.org/euclid.cag/1175790107},
			review={\MR{2273294}},
		}
		
		\bib{DaminYau2018}{article}{
			author={Wu, Damin},
			author={Yau, Shing-Tung},
			title={Complete k\"{a}hler-Einstein metrics under certain holomorphic
				covering and examples},
			date={2018},
			journal={Ann. Inst. Fourier (Grenoble)},
			volume={68},
			number={7},
			pages={2901\ndash 2921},
		}
		
		\bib{DaminYau2017}{article}{
			author={Wu, Damin},
			author={Yau, Shing-Tung},
			title={Invariant metrics on negatively pinched complete k\"{a}hler
				manifolds},
			journal={J. Amer. Math. Soc.},
			date={2020},
			volume={33},
			pages={103\ndash 133  },
		}
		
		\bib{YauReviewOfGeometry}{incollection}{
			author={Yau, S.-T.},
			title={Review of geometry and analysis},
			date={2000},
			booktitle={Mathematics: frontiers and perspectives},
			publisher={Amer. Math. Soc., Providence, RI},
			pages={353\ndash 401},
			review={\MR{1754787}},
		}
		
		\bib{Yau1978}{article}{
			author={Yau, Shing-Tung},
			title={M\'{e}triques de k\"{a}hler-einstein sur les vari\'{e}t\'{e}s
				ouvertes},
			date={1978},
			journal={In (Premi\'{e}re Classe de Chern et courbure de Ricci: Preuve de
				la conjecture de Calabi)},
			volume={volume 58 of S\'{e}minaire Palaiseau},
			pages={163\ndash 167},
		}
		
		\bib{YauZhang2014}{article}{
			author={Yau, Shing-Tung},
			author={Zhang, Yi},
			title={The geometry on smooth toroidal compactifications of {S}iegel
				varieties},
			date={2014},
			ISSN={0002-9327},
			journal={Amer. J. Math.},
			volume={136},
			number={4},
			pages={859\ndash 941},
			url={https://doi-org.ezproxy.lib.uconn.edu/10.1353/ajm.2014.0024},
			review={\MR{3245183}},
		}
		
	\end{biblist}
\end{bibdiv}

\end{document}